\documentclass[leqno]{article}%{dcds-bOF}
\usepackage{hyperref}\hypersetup{hidelinks}
\usepackage{amsmath,amsfonts,amssymb,amsthm}
\usepackage{graphicx}
\usepackage{enumerate}
\usepackage{bm} % boldmath
\usepackage{array}
\usepackage{multicol}
\usepackage{subcaption}
\usepackage{placeins}
\usepackage{amsopn}
\usepackage{cleveref}

\newtheorem{theorem}{Theorem}[section]

\newtheorem{lemma}[theorem]{Lemma}
\newtheorem{proposition}{Proposition}

\newtheorem{remark}{Remark}
\title{Dynamics of a
  prey-predator system
	with modified Leslie-Gower and Holling type II schemes incorporating
	a prey refuge}
      \date{}
      \author{}

\newcommand\tq{;\,} %% tel que
\newcommand\R{\mathbb R}%{\bf R}
\renewcommand\d{\mathrm{d}} % différentielle
\newcommand\dd{\,\d}

\newcommand\abs[1]{\left\vert{#1}\right\vert}
\newcommand\accol[1]{\left\{{#1}\right\}}
\newcommand\CCO[1]{\left({#1}\right)}
\newcommand\croche[1]{\left[{#1}\right]}

\newcommand\un[1]{\,\rlap{{1}}\kern.22em \mbox{l}_{#1}} %% fonction
\newcommand\jcb[1]{\mathcal{J}(#1)} % jacobien
\newcommand\trace{\mathop{\mathrm{Trace}}}

\newcommand\invrg{\mathcal{A}} % région invariante
\newcommand\invrgB{\mathcal{B}} %
\newcommand\tangency{r} % ordonnée d'un point de tangence

\newcommand\crit[1]{{{#1}^*}}
\newcommand\Ec{\crit{E}}
\newcommand\xc{\crit{x}}
\newcommand\yc{\crit{y}}
\newcommand\zc{\kappa}

\newcommand\yd{{\tilde{y}}}

\newcommand\xd{l} % abscisse du point d'intersection de y=k_2 et y=U(x)
\newcommand\rectangle{\mathcal{R}}

\newcommand\gcrit[1]{#1}%{{\underline{#1}}}
\newcommand\xx{\gcrit{x}}
\newcommand\yy{\gcrit{y}}
\newcommand\EE{\gcrit{E}}

\newcommand\XX{{X}} % coordonnées centrées autour d'un point critique
\newcommand\YY{{Y}}
\newcommand\cc{\rho}%%{\mathfrak{c}} % coefficients
\newcommand\KX{{\mathfrak{u}}}%coordonnées centrées autour d'un point critique
\newcommand\KY{{\mathfrak{v}}}
\newcommand\Kf{{\mathfrak{f}}}
\newcommand\gfun{{\mathfrak{g}}}
\newcommand\Cf{{\bm{f}}}
\newcommand\Cg{{\bm{g}}}
\newcommand\KFF{F}
\newcommand\KGG{G}
\newcommand\Kc{c} % =a\cc

\newcommand\vectorfield{\bm{v}} % necessite le paquetage bm
\newcommand\vfA{{\mathcal{V}_1}}%{f} % coordonnées de $\vectorfield$
\newcommand\vfB{{\mathcal{V}_2}}%{g}

\newcommand\ii{\bm{i}}
\newcommand\jj{\bm{j}}
\newcommand\uu{\bm{u}}
\newcommand\vv{\bm{v}}
\newcommand\dzo{\delta}
\newcommand\DZO{\gamma}

\newcommand\xpmin{\underline{x}}
\newcommand\xpmax{\overline{x}}
\newcommand\RR{
4\alpha_{{2}}^{3}\alpha_{{0}}
-\alpha_{{2}}^{2}\alpha_{{1}}^{2}
+4\alpha_{{1}}^{3}
-18\alpha_{{2}}\alpha_{{1}}\alpha_{{0}}
+27\alpha_{{0}}^{2}
}

\newcommand\roh{\rho}
\newcommand\Pf{P}
\newcommand\Pg{Q}

\newcommand\tk{t_{0}}

\newcommand\TTI[1]{\textrm{\bf #1}}
\newcommand\xxi{\TTI{x}}
\newcommand\yyi{\TTI{y}}

\newcommand\DD{D}

\newcommand\nroot{\mathfrak{p}} % nb de racines de partie reelle positives
\newcommand\nrootp{\mathfrak{n}} % nb de racines positives distinctes
\newcommand\nrootd{\nrootp}%{\mathfrak{k}} % nb de racines positives
\newcommand\cf{v_1}%{f}
\newcommand\cg{v_2}%{g}
\newcommand\Dulac{D}

\newcommand\charac{\chi} \newcommand\discrimC{\Delta_\charac}
\newcommand\discrimRp{\Delta_{R'}}
\newcommand\discrimQ{\Delta}
\newcommand\tribu{\mathcal{F}}
\DeclareMathOperator{\prob}{\mathrm{P}}
\DeclareMathOperator{\expect}{\mathrm{E}}

\newcommand\xu{u}%{\hat{x}}%{x^{\mathrm{u}}}
\newcommand\yu{v}%{\hat{y}}%{y^{\mathrm{u}}}
\newcommand\xdd{\check{u}}%{\hat{x}}%{x^{\mathrm{u}}}
\newcommand\ydd{\check{v}}%{\hat{y}}%{y^{\mathrm{u}}}
\newcommand\txu[1]{{\tau}_1^{(#1)}}
\newcommand\tyu[1]{{\tau}_2^{(#1)}}
\newcommand\txdd[1]{\check{\tau}_1^{(#1)}}
\newcommand\tydd[1]{\check{\tau}_2^{(#1)}}

\allowdisplaybreaks
\numberwithin{equation}{section}
\begin{document}
	\maketitle

% Enter the first author's name and address:
\centerline{\scshape Safia Slimani\footnote{Supported by TASSILI research program 16MDU972 between the University of Annaba (Algeria) and the University of Rouen (France)}}
\medskip
{\footnotesize
% please put the address of the first author
 \centerline{Normandie Univ, Laboratoire Rapha\"el Salem}
 \centerline{UMR CNRS 6085, Rouen, France% \\
 % \textit{email address}: {slimani\_safia@yahoo.fr}
}
   } % Do not forget to end the {\footnotesize by the sign }

\medskip

\centerline{\scshape Paul Raynaud de Fitte and Islam Boussaada}
\medskip
{\footnotesize
 % please put the address of the second  and third author
 \centerline{Normandie Univ, Laboratoire Rapha\"el Salem,}
 \centerline{UMR CNRS 6085, Rouen, France.}
 \centerline{PSA $\&$ Inria DISCO $\&$ Laboratoire des Signaux et
   Syst\`emes,}   \centerline{Universit\'e Paris Saclay,
   CNRS-CentraleSup\'elec-Universit\'e Paris Sud,}
\centerline{3 rue Joliot-Curie, 91192 Gif-sur-Yvette cedex, France}
% \centerline{Laboratoire des Signaux et Syst\`emes (L2S)
%      Sup\'elec-CNRS-Universit\'e Paris Sud, France, }
%    \centerline{and
% Laboratoire de Mod\'elisation et Calcul Scientifique (LMCS)}
%    \centerline{Institut Polytechnique des Sciences Avanc\'ees, France.}
}

\bigskip

% The name of the associate editor will be entered by an editorial staff
% "Communicated by the associate editor name" is not needed for special issue.
 %\centerline{(Communicated by the associate editor name)}

	\begin{abstract}
		We study a modified version of a prey-predator system
                with modified Leslie-Gower and Holling type II
                functional responses studied by M.A.~Aziz-Alaoui and
                M.~Daher-Okiye. The modification consists in
                incorporating a refuge for preys, and substantially
                complicates the dynamics of the system.
                %%%%%%%%%%%%%%%
                We study the local and global dynamics and
                the existence of cycles.
                %%%%%%%%%%%%
                We also
                investigate conditions for extinction or
                existence of a stationary
                distribution, in the case of a stochastic perturbation
                of the system.
	\end{abstract}

        
        \newpage
	\tableofcontents

\section{Introduction}
    We study a two-dimensional prey-predator system
    with modified Leslie-Gower and Holling type II
    functional responses.
    This system is a generalization of the system investigated in the
papers by
    M.A.~Aziz-Alaoui and M.~Daher-Okiye
\cite{Daher-Aziz03,Daher-Aziz03bologn}.

    Aziz-Alaoui and Daher-Okiye's model has been studied and
    generalized in numerous papers:
    models with spatial diffusion term \cite{camara2011waves,yafia1,yafia2,yafia3},
        with time delay \cite{nindjin,yafia,yafia0}, with stochastic perturbations
\cite{lv-wang2011asymptotic,lv-wang2012analysis,mandal-banerjee2013,liu2014stochastic_holling},
        or incorportaing a refuge for the prey
    \cite{chen-chen-xie2009refuge}, to cite but a few.

    A novelty of the present paper
    is that we add a refuge
        in a way which is different from
        \cite{chen-chen-xie2009refuge}, since the density of prey in our
    refuge
    is not proportional
    to the total density of prey.
    This kind of refuge entails a qualitatively different
    behavior of the solutions, even for a small refuge,
    contrarily to the type of refuge investigated in
    \cite{chen-chen-xie2009refuge}.  Let us emphasize that, even in the
case without refuge, our study provides new results.

In the first and main part of the paper (%Section
\Cref{sec:determinist}),
        we study the system of
	\cite{Daher-Aziz03,Daher-Aziz03bologn} with refuge, but without
	stochastic perturbation:
	\begin{equation}
		\label{eq:general}
		\left\{\begin{aligned}
			&\dot{\xxi}=\xxi(\rho_1-\beta \xxi)
		-\frac{\alpha_1\yyi(\xxi-\mu)_+}{\kappa_1+(\xxi-\mu)_+},\\
			&\dot{\yyi}=\yyi\CCO{\rho_2-\frac{\alpha_2\yyi}{\kappa_2+(\xxi-\mu)_+} }.
		\end{aligned}\right.
	\end{equation}
	In this system,
	\begin{itemize}
		\item $\xxi\geq 0$ is the density of prey,
		\item $\yyi\geq 0$ is the density of predator,
		\item $\mu\geq 0$ models a refuge for the prey, i.e,
		the quantity $(\xxi-\mu)_+:=\max(0,\xxi-\mu)$ is
		the density of prey which is accessible to the predator,
		\item $\rho_1>0$ (resp.~$\rho_2>0$) is the growth rate
                  of prey (resp.~of predator),
		\item $\beta>0$ measures the strength of competition
                  among individuals of the
		prey species,
		\item $\alpha_1>0$ (resp.~$\alpha_2>0$) is the rate of
		reduction of preys
		(resp.~of predators)
		\item $\kappa_1>0$ (resp.~$\kappa_2>0$) measures the
                  extent to which
		the  environment provides protection to the prey
                (resp.~to the predator).
	\end{itemize}
	When the predator is absent, the density of prey
        $\xxi$ satisfies a logistic equation and
	converges to $\frac{\rho_1}{\beta}$, so we assume that
	$$0\leq\mu<\frac{\rho_1}{\beta}.$$
%%%%%%%%%%%%%%%%%%%
        The last term in the right hand side of the first equation
          of \eqref{eq:general},
  which expresses the loss of prey population due to the predation,
  is a modified Holling type II functional response, where the
  modification consists in the introduction of the refuge $\mu$.
  The predation rate of the predators decreases when they are driven to
  satiety, so that the consumption rate of preys decreases
  when the density of prey increases.
  % Lajmiri This response function which models
% the consumption of preys by predators is such that the predation rate of predators
% increases when the preys are few and decreases when they reach their satiety.

  Similarly, if its favorite prey is absent (or hidden in the refuge),
  the predator has a
  logistic dynamic, which means that it
survives with other prey species, but with limited growth.
  The last term in the right hand side of the second equation,
  of \eqref{eq:general} is a modified Leslie-Gower functional
  response, see \cite{Leslie-Gower, pielou}.
  % \cite{Leslie-Gower} and \cite[page 91]{pielou}.
  Here,
  the modification lies in the addition of the constant $\kappa_2$, as
  in \cite{Daher-Aziz03,Daher-Aziz03bologn}, as well as in
the introduction of the refuge $\mu$. It models the loss of
predator population when the prey becomes less available, due its
rarity and the refuge.
%Yang2013 If this
% favorite food is lacking severely, the predator y will switch to other populations, but its growth will be limited.
%%%%%%%%%%%%%%%%%%%

	Setting, for $i=1,2$,
	\begin{gather*}
		%\label{eq:simplification}
		%t=\rho_1\tau,\
		x(t)=\frac{\beta}{\rho_1}\,\xxi\CCO{\frac{t}{\rho_1}},
		\ y(t)=\frac{\beta}{\rho_1}\,\yyi\CCO{\frac{t}{\rho_1}},\\
		m=\frac{\mu\beta}{\rho_1},\
		a=\frac{\alpha_1\rho_2}{\alpha_2\rho_1},\
		k_i=\frac{\kappa_i\beta}{\rho_1},\
		b=\frac{\rho_2}{\rho_1},
	\end{gather*}
	we get the simpler equivalent system
	\begin{equation}
		\label{eq:simple}
		\left\{\begin{aligned}
			&\dot{x}=x(1-x)-\frac{a y(x-m)_+}{k_1+(x-m)_+},\\
			&\dot{y}=b y\CCO{1-\frac{y}{k_2+(x-m)_+} },
		\end{aligned}\right.
	\end{equation}
	where $0\leq m<1$, all other parameters are positive,
	and $(x,y)$ takes its values in the quadrant
        $\R_+\times\R_+$.

        In this first part, we study the dynamics of
        \Cref{eq:simple}, which is complicated by
        the refuge parameter $m$.
        However, even in the case when $m=0$,
        we provide some new results.
        We first show the
        persistence and the existence of a compact attracting set.
        Then,
        we study in detail the equilibrium points (there can be 3 distinct
        non trivial such points when $m>0$) and their local
        stability.
        We also give sufficient conditions for
        the existence of a globally
        asymptotically stable equilibrium, and  we give some
        sufficient conditions for the absence of periodic orbits.
        A stable limit cycle may surround several limit points, as we
        show numerically.

	In a second part (%Section
\Cref{sec:stoch}),
        we study the stochastically perturbed system
	\begin{equation} \label{eq:stochastic}
	\left\{\begin{aligned}
dx(t)&=\CCO{x(t)(1-x(t))-\frac{a y(t)(x(t)-m)_+}{k_1+(x(t)-m)_+}}dt
        + \sigma_{1} x(t) dw_{1}(t),\\
dy(t)&=b y(t)\CCO{1-\frac{y(t)}{k_2+(x(t)-m)_+}}dt+\sigma_{2} y(t) dw_{2}(t),
	\end{aligned}\right.
	\end{equation}
	where $w=(w_1,w_2)$ is a standard Brownian motion defined on the
	filtered probability space
	$(\Omega,\tribu,(\tribu_{t}),\prob )$,
	and $\sigma_1$ and $\sigma_2$ are constant real numbers.
%%%%%%%%%%%%%%%%%%%%%%%%%%%%%%%
        This perturbation represents the environmental
          fluctuations. There are many ways
          to model the randomness of the environment, for
          example using random parameters in \Cref{eq:simple}.
          Since the right hand side of \Cref{eq:simple}
          depends nonlinearly on many parameters,
          the approach using It\^o stochastic differential equations
          with Gaussian centered noise
          models in a simpler way
          the fuzzyness of the solutions.
          % around their mean values,
          % that is, the
          % values taken by the deterministic system.
          The
          choice of a multiplicative noise in this context is
          classical, see \cite{may},
          and it has the great advantage over additive noise
            that solutions starting in the quadrant
            $[0,+\infty[\times[0,+\infty[$ remain in it.
          % Since the values of $x$ and $y$ are to take positive values,
          % it is also
          % natural to assume that small values of $x$ and/or $y$ are
          % less perturbed, so that
          % the noise has to be proportional to the values of $x$ and $y$.
          Furthermore, the independence of the Brownian motions $w_1$
          and $w_2$ reflects the independence of the parameters
          in both equations of \eqref{eq:simple}.
          % Of course, Stratonovich equations could also be considered,
          % as well as coloured noise instead of white noise.
          %It\^o vs Stratonovich
          %This choice is also based on simplicity.

          Another possible choice of stochastic perturbation
          would be to center the
          noise on an equilibrium point of the
          deterministic system, as in \cite{BC}. But we shall see
          in \Cref{theo:nbdepoints} that
          \Cref{eq:simple} may have three distinct equilibrium
          points. Furthermore,
          as in the case of additive noise,
          this type of noise
          would allow the solutions to have excursions outside the
          quadrant $[0,+\infty[\times[0,+\infty[$,
          which of course would be unrealistic.
       %%%%%%%%%%%%%%%%%%%%%%%%%%%%%%%%

    %In this second part of the paper,
       We show in \Cref{sec:stoch}
       the existence and uniqueness of the global positive
        solution with any initial positive value of the stochastic
        system \eqref{eq:stochastic},
        and we show that, when the diffusion coefficients
        $\sigma_1>0$ and $\sigma_2>0$ are small,
        the solutions to \eqref{eq:stochastic} converge to a unique
        ergodic stationary distribution, whereas, when they are large,
        the system \eqref{eq:stochastic} goes asymptotically to
        extinction.
        Small values of $\sigma_1$ and
          $\sigma_2$ are more interesting for ecological modeling,
          because they make solutions of \eqref{eq:stochastic}
          closer to the prey-predator dynamics.
        The effect of such a small or moderate perturbation is the
        disparition of all equilibrium points of the open quadrant
        $]0,+\infty[\times]0,+\infty[$, replaced by a unique
        equilibrium, the stationary ergodic distribution, which is an
        attractor.

      % Finally, in \Cref{sec:simul},
      The last part of the paper is \Cref{sec:simul}, where
        we make numerical simulation to illustrate our results.

	%%%%%%%%%%%%%%%%%%%%%%%%%%%%%%%%%%%%%%%%%%%%%%%%%%%%%%%%%%%%%%%%%%
	\section{Dynamics of the deterministic system}
        \label{sec:determinist}
	In this section, we study the dynamics of \eqref{eq:simple}.

        Throughout, we denote by $\vectorfield$ the vector field
	associated with \eqref{eq:simple}, and
         $$\vectorfield=v_1\frac{\partial}{\partial x}
                       +v_2\frac{\partial}{\partial y},$$
        so that \eqref{eq:simple} reduces to $\bigl(\dot{x}=v_1$ and
        $\dot{y}=v_2\bigr)$.

	The right hand side of \eqref{eq:simple} is locally Lipschitz,
	thus, for any initial condition, \eqref{eq:simple} has a unique
	solution defined on a maximal time interval.
	
	Furthermore, the axes are invariant manifolds of
	\eqref{eq:simple}:
	\begin{itemize}
		\item If $x(0)=0$, then $x(t)=0$ for every $t$, and $\dot{y}=by(1-y/k_2)$ yields
		$$y(t)=\frac{y(0)k_2}{k_2+y(0)(e^{bt}-1)},$$
		thus $\lim_{t\rightarrow+\infty}y(t)=k_2$ if $y(0)>0$.
		
		\item If $y(0)=0$, then $y(t)=0$ for every $t$, and $\dot{x}=x(1-x)$ yields
		$$x(t)=\frac{x(0)}{1+x(0)(e^{t}-1)},$$
		thus $\lim_{t\rightarrow+\infty}x(t)=1$ if $x(0)>0$.
	\end{itemize}
	From the uniqueness theorem for ODEs, we deduce that the
	open quadrant $]0,+\infty[\times ]0,+\infty[ $ is stable,
	thus there is no extinction of any species in finite time.

	\subsection{Persistence and compact attracting set}%
	\label{subsec:persistence}
	
	% compact trapping domain ? (terminologie de Broer et al, 2007)
	
        The next result shows that there is no explosion of the system \eqref{eq:simple}.
	It also shows a qualitative difference brought by the refuge:
	when $m=0$, the density of prey
	may converge to $0$, whereas, when $m>0$, the system \eqref{eq:simple}
	is always uniformly persistent. % in the positive quadrant.
		
	Let %us show that the compact set
	$$\invrg=\accol{(x,y)\in\R^2 \tq m\leq x\leq 1,\ k_2\leq y< L},$$
	where $L=1+k_2-m$.
	%Our first result shows a qualitative change induced by the refuge.
	%%%%%%%%%%%%%%%%%%%% PROPOSITION A INVARIANT %%%%
	\begin{theorem}\label{theo:Ainvariant}
		\begin{enumerate}[(a)]
			\item\label{item:invariant}
			The set $\invrg$ is invariant for \eqref{eq:simple}.
			Furthermore, if the initial condition  $(x(0),y(0))$ is in
			the open quadrant $]0,+\infty[\times
                        ]0,+\infty[ $, we have
                        \begin{equation}
				\label{eq:limsupinf}
				\left\{ \begin{aligned}
					m&\leq \liminf_{t\rightarrow+\infty}x(t)
					\leq \limsup_{t\rightarrow+\infty}x(t)
					\leq 1,\\
					k_2&\leq \liminf_{t\rightarrow+\infty}y(t)
					\leq \limsup_{t\rightarrow+\infty}y(t)
					\leq L.
				\end{aligned} \right.
			\end{equation}

			\item\label{item:persistent}
			In the case when $m>0$, for any initial
                        condition  $(x(0),y(0))$ in
			the open quadrant $]0,+\infty[\times
                        ]0,+\infty[ $,
                        the solution $(x(t),y(t))$ enters $\invrg$ in finite time.
                        In particular, the system \eqref{eq:simple} is
                        uniformly persistent.

			\item\label{item:Am0}
			In the case when $m=0$, for any $\epsilon>0$ such that
                        $k_2-\epsilon>0$,
                        the compact set	
                        $[0,1]\times[k_2-\epsilon,L]$ is invariant,
                        and,
                        for any initial
                        condition  $(x(0),y(0))$ in
			the open quadrant $]0,+\infty[\times]0,+\infty[ $,
			the solution $(x(t),y(t))$ enters
                        $[0,1]\times[k_2-\epsilon,L]$ in finite time.
                        Furthermore:
                        \begin{enumerate}[(i)]
                          \item\label{item:Am0-i} If $aL<k_1$,
                           the system \eqref{eq:simple} is
			   uniformly persistent.
                        More precisely,
                        %for any initial condition
                        if $(x(0),y(0))\in]0,+\infty[\times]0,+\infty[$, we have
			\begin{equation}
		           \label{eq:m0liminf}
		           \liminf_{t\rightarrow+\infty}x(t)\geq \frac{k_1-aL}{k_1}.
			%\text{ and }\liminf_{t\rightarrow+\infty}y(t)=k_2.
			\end{equation}

                          \item\label{item:Am0-ii} If $ak_2< k_1\leq aL$,
                                %there is no extinction.
                                the system \eqref{eq:simple} is
                                 uniformly weakly
                                  persistent. More precisely, if
                                  % any initial condition
                        $(x(0),y(0))\in]0,+\infty[\times]0,+\infty[$, we have
			\begin{equation}%{multline}
		           \label{eq:m0limsup}
		   \limsup_{t\rightarrow+\infty}x(t)%\\
                       \geq
                         \min\CCO{\frac{k_1}{a}-k_2,
                     \frac{1-k_1-a+\sqrt{(1-k_1-a)^2+4(k_1-ak_2)}}{2}
                                  }.
			\end{equation}%{multline}
                          \item\label{item:Am0-iii-special}
                          If $k_1=ak_2$, then:
                          \begin{itemize}
                          \item If $1-k_1-a>0$, %$(1-k_1)/a> 1$,
                          the system \eqref{eq:simple} is
                               uniformly  weakly persistent.
                          More precisely, if
                        $(x(0),y(0))\in]0,+\infty[\times]0,+\infty[$, we have
                        \begin{equation}
                          \label{eq:limsup-ak2=k1}
                          \limsup_{t\rightarrow+\infty}x(t)\geq 1-k_1-a.
                        \end{equation}

                          \item If $1-k_1-a\leq 0$, %$(1-k_1)/a\leq 1$,
                            the point $E_2=(0,k_2)$ is globally
                            attracting, thus the prey becomes
                            extinct asymptotically for any initial
                            condition in $]0,+\infty[\times]0,+\infty[$.

                          \end{itemize}

                          \item\label{item:Am0-iii} If $k_1< ak_2$, the
                            point $E_2=(0,k_2)$ is globally
                            attracting, thus the prey becomes
                            extinct in infinite time for any initial
                            condition in $]0,+\infty[\times]0,+\infty[$.

                        \end{enumerate}
		\end{enumerate}
	\end{theorem}
	%%%%%%%%%%
\begin{remark}
A more general sufficient condition of global attractivity of $E_2$
is provided by \Cref{prop:m=0nbdepoints} (see %Remark
\Cref{rem:+precis}).
\end{remark}

        %%%%%%%%%%
%\preuvof{Theorem \cref{theo:Ainvariant}}
\begin{proof}[Proof of \Cref{theo:Ainvariant}]
	\eqref{item:invariant}
        % {\em First step}
        % Let us first prove that $\invrg$ is invariant and that
        % \eqref{eq:limsupinf} is satisfied.
	When $m=0$, the first inequality in \eqref{eq:limsupinf} is trivial.
	In the case when $m>0$, we need to prove that
	$\liminf x(t)\geq m$, provided that
	$x(0)>0$. Actually we have a better result, since,
	if $x(0)\leq m$, then $x$
	coincides with the solution to the logistic
	equation
	$\dot{x}=x(1-x)$ as long as $x$ does not reach the value $m$, that is,
	$$x(t)=\frac{x(0)e^{t}}{1+x(0)(e^{t}-1)}.$$
	If $x(0)>0$, this function converges to $1$,
	thus there exists $t_m>0$ such that
	\begin{equation}
		\label{eq:tm}
		t\geq t_m\Rightarrow x(t)\geq m.
	\end{equation}
	% $x$ reaches
	% the value $m$ in finite time.
	Note that, when $m>0$, if $x(t)=m$, we have $\dot{x}(t)=m(1-m)>0$.
	Thus
	\begin{equation}
		\label{eq:x=m}
		\left\lgroup x(0)\geq m\right\rgroup \Rightarrow \left\lgroup x(t)\geq
		m,\ \forall t\geq 0\right\rgroup,
	\end{equation}
	which implies the first inequality in \eqref{eq:limsupinf}.
% When $m>0$, note also that, for $x=m$, we have
	% $\dot{x}=x(1-x)=m(1-m)>0$ thus, once $(x,y)$ is in $\invrg$, it
	% stays in $\invrg$.
	Now, from the first equation of \eqref{eq:simple}, we have
	\begin{equation*}%\label{eq:}
		\dot{x}\leq x(1-x),
	\end{equation*}
	which implies that, for every $t\geq 0$,
	\begin{equation}
		\label{eq:xlog+}
		x(t)\leq \frac{x(0)e^{t}}{1+x(0)(e^{t}-1)}.
	\end{equation}
	In particular, we have
	\begin{equation}
		\label{eq:xleq1}
		\limsup_{t\rightarrow+\infty}x(t)\leq 1
		\text{ and }
		\left\lgroup x(0)\leq 1\Rightarrow x(t)\leq 1,\
                       \forall t\geq 0 \right\rgroup.
	\end{equation}
	This implies that, for any $\epsilon>0$, and for $t$ large enough
	(depending on $x(0)$), we have $x(t)\leq 1+\epsilon$. We deduce that,
	for any $\epsilon>0$, and for $t$ large enough, we have
	\begin{equation}
		\label{eq:inegalites_ypoint}
		b y\CCO{1-\frac{y}{k_2} }
		\leq \dot{y}(t)\leq
		b y\CCO{1-\frac{y}{k_2+1+\epsilon-m} }
		=b y\CCO{1-\frac{y}{L+\epsilon} },
	\end{equation}
	which implies that, for $t$ large enough, say, $t\geq t_0$,
	\begin{equation}
		\label{eq:inegalites_y}
		\frac{y(0)k_2e^{bt}}{k_2+y(0)(e^{bt}-1)}
		\leq y(t) \leq
		\frac{y(t_0)(L+\epsilon)e^{b(t-t_0)}}{L+\epsilon+y(t_0)(e^{b(t-t_0)}-1)}.
	\end{equation}
	Of course, if $x(0)\leq 1$, we can drop $\epsilon$ in
	\eqref{eq:inegalites_ypoint} and \eqref{eq:inegalites_y}.
	Thus, we have
	\begin{equation}
		\label{eq:inegalites_yA}
		\left\lgroup x(0)\leq 1 \text{ and }k_2\leq y(0)\leq L\right\rgroup
		\Rightarrow
		\left\lgroup k_2\leq y(t)\leq L,\ \forall t\geq 0\right\rgroup.
	\end{equation}
	We deduce from  \eqref{eq:x=m}, \eqref{eq:xleq1}, and
        \eqref{eq:inegalites_yA} that
	$\invrg$ is invariant.
	
	As $\epsilon$ is arbitrary in \eqref{eq:inegalites_y},
        we have also, when $y(0)>0$,
	\begin{equation}\label{limites_y}
		k_2\leq \liminf_{t\rightarrow +\infty}y(t)
		\leq \limsup_{t\rightarrow +\infty}y(t)
		\leq L.
	\end{equation}
	From \eqref{eq:tm}, \eqref{eq:xleq1},
	and \eqref{limites_y},
	we deduce \eqref{eq:limsupinf}.

%%%%%%%%%%%%%%%%%%%% (b) persistence for m>0
	\bigskip
	\eqref{item:persistent}
        We have already seen that $x(t)\geq m$ for $t$
        large enough, let us now check that $x(t)\leq 1$ for $t$ large
        enough. Since $\invrg$ is invariant, we only need to prove
        this for $x(0)> 1$.
        Let $\epsilon>0$ such that $k_2-\epsilon>0$.
        Let $\delta>0$ such that $\delta+m<1$ and such that
        \begin{equation}
          \label{eq:xdelta}
          (x\geq 1-\delta)\Rightarrow
           x(1-x)< \frac{a(k_2-\epsilon)(1-m)}{1+\epsilon-m}.
        \end{equation}
        From the first inequality in
        \eqref{limites_y}, we have $y(t)\geq k_2-\epsilon$
        for $t$ large enough, say $t\geq \tk$.
        From \eqref{eq:xleq1}, we can take $\tk$ large
        enough such that, for $t\geq\tk$, we have also $x(t)\leq
        1+\epsilon$.
        Using \eqref{eq:xdelta},
        we deduce, for $t\geq \tk$ and $x(t)\geq 1-\delta$,
        \begin{align*}
          \dot{x}(t)&\leq
          x(t)\bigl(1-x(t)\bigr)
              -\frac{a(k_2-\epsilon)(1-\delta-m)}{1+\epsilon-m}\\
          &\leq -\frac{a\delta (k_2-\epsilon)}{1+\epsilon-m}.
        \end{align*}
        Thus $x$ decreases with speed less than $-\frac{a\delta (k_2-\epsilon)}{1+\epsilon-m}<0.$
	Thus $x(t)\leq 1-\delta$ for $t$ large enough.

        We can now repeat the reasoning of
       \eqref{eq:inegalites_ypoint} and \eqref{eq:inegalites_y},
        replacing $\epsilon$ by $-\delta$, which yields that
        $\limsup y(t)\leq L-\delta$. In
        particular, $y(t)<L$ for $t$ large enough.
       %%%%%%%%%%%%%%%%

	% Assume that $m>0$, and let us prove that
	% $(x(t),y(t))$ enters $[m,\infty[\times[k_2,+\infty[$ in
        % finite time.
        To prove that $y(t)>k_2$ for $t$ large enough, let us first
        sharpen the result of \eqref{eq:tm}. This is where we use that
        $m>0$.
	Let $\delta>0$, with $m+\delta<1$.
	If $\abs{x-m}<\delta$, we have
	\begin{equation*}
		\abs{x(1-x)-m(1-m)}=\abs{(x-m)\CCO{1-(x+m)}}\leq \abs{x-m}<\delta.
        \end{equation*}
 %%%%%%%%%%%%%%%%%%%%%%%%%%%%%%%%
	From \eqref{limites_y}, %{eq:inegalites_y},
          we deduce
          that, for any $\epsilon>0$, and $t$ large enough,
          depending on $\epsilon$, we have
          $$y(t)\leq L+\epsilon\text{ and }x(t)\geq m,$$
          from which we deduce
	\begin{equation*}
		\dot{x}\geq x(1-x)-\frac{a(L+\epsilon)\delta}{k_1}
		\geq \DD:=m(1-m)-\delta-\frac{a(L+\epsilon)\delta}{k_1}.
              \end{equation*}
	(we do not write $t$ here for the sake of simplicity).
	For $\delta$ small enough, we have $\DD>0$.
	Thus, if $m>0$, we can find $\delta>0$ small enough (depending on
	$m$), such that, when $x(t)$ is in the interval
        $[m,m+\delta]$, it reaches the
	value $m+\delta$ in finite time (at most $\DD\delta$),
	and then it stays in $[m+\delta,1]$.
	Using \eqref{eq:tm}, we deduce that there exists
        $t_{m+\delta}>0$ such that
	\begin{equation}
		\label{eq:tmdelta}
		t\geq t_{m+\delta}\Rightarrow x(t)\geq m+\delta.
	\end{equation}
	Using \eqref{eq:tmdelta} in \eqref{eq:simple}, we obtain, for
        $t\geq t_{m+\delta}$,
	\begin{equation*}
		%\label{eq:ydelta}
		\dot{y}\geq b y\CCO{1-\frac{y}{k_2+\delta} },
	\end{equation*}
	which yields, if $y(0)>0$,
	\begin{equation*}
		%\label{eq:y>k2}
		y(t)\geq
		\frac{y(t_{m+\delta})(k_2+\delta)e^{b(t-t_{m+\delta})}}{k_2+\delta+y(t_{m+\delta})(e^{b(t-t_{m+\delta})}-1)}.
	\end{equation*}
	This proves that
	\begin{equation*}
		%\label{eq:ydelta}
		\liminf_{t\rightarrow+\infty}y(t)\geq k_2+\delta,
	\end{equation*}
	and that $y>k_2$ for $t$ large enough.

%%%%%%%%%%%%%%%%%%%%%%%%%%%%%%% (c) cas m=0	
	\bigskip
	\eqref{item:Am0}
	Assume now that $m=0$.
Since the first part of the proof of \eqref{item:persistent}
is valid for all
$m\geq 0$, we have already proved that
$x(t)< 1$ and $y(t)<L$ for $t$ large enough.
Let $\epsilon>0$ such that $k_2-\epsilon>0$. For $y<k_2$, we have
$\dot{y}>0$, thus $[0,1]\times[k_2-\epsilon,L]$ is invariant.
Furthermore,  for any initial condition
$(x(0),y(0))\in]0,+\infty[\times]0,+\infty[$,
since $\liminf y(t)\geq k_2$,
we have $y(t)>k_2-\epsilon$ for $t$ large enough, thus
$(x(t),y(t))$ enters $[0,1]\times[k_2-\epsilon,L]$ in finite time.

        %%%%%%%%% (i)
        \medskip
        \noindent\eqref{item:Am0-i}
	Assume that $aL<k_1$, and let $\epsilon>0$ 0
        such that $a(L+\epsilon)<k_1$. Let
	$K_\epsilon=\frac{k_1-a(L+\epsilon)}{k_1}$.
	By the second inequality in \eqref{limites_y},
	we have, for $t$ large enough
	\begin{equation}
		\label{eq:m0xgeq}
		\dot{x}\geq x(1-x)-\frac{ax(L+\epsilon)}{k_1}
		=K_\epsilon x\CCO{1-\frac{x}{K_\epsilon}}.
	\end{equation}
	Thus $\liminf x(t)\geq K_\epsilon$.
	As $\epsilon$ is arbitrary, this proves \eqref{eq:m0liminf}.
	From \eqref{eq:m0liminf} and the first inequality in
	\eqref{limites_y}, we deduce that \eqref{eq:simple} is uniformly
	persistent.
	
        %%%%%%%%% (ii) uniform weak persistence
        \medskip
        \noindent\eqref{item:Am0-ii}
Assume now that $ak_2< k_1\leq aL$.
%%%
Observe first that,
if $\limsup x(t)<\xd$ for some
$\xd>0$, then, for $t$ large enough, we have $x(t)<\xd$,
thus $\dot{y}(t)<by(1-y/(k_2+l))$.
We deduce that
\begin{equation}\label{eq:k2+l}
  \limsup_{t\rightarrow\infty}x(t)<\xd
\Rightarrow
  \limsup_{t\rightarrow\infty} y(t)<k_2+\xd.
\end{equation}
Let us now rewrite the first equation of \eqref{eq:simple} as
\begin{equation*}
\dot{x}=x\left(1-x-\frac{ay}{k_1+x}\right)
    =\frac{x}{k_1+x}\Bigl(-(x-1)(x+k_1)-ay \Bigr),
\end{equation*}
that is,
\begin{equation}
  \label{eq:xm=0}
 \dot{x}=\frac{ax}{k_1+x}\Bigl(U(x)-y \Bigr)
\end{equation}
where $U(x)=(-1/a)(x-1)(x+k_1)$.
Since $ak_2<k_1$, the point $E_2$ lies below the parabola $y=U(x)$, thus
in the neighborhood of $E_2$, for $x>0$, we have $\dot{x}>0$.
%%%%

By \eqref{eq:k2+l},
if $\limsup x(t)<\xd$ for some
$\xd>0$,
then
for $t$ large enough,
the point $(x(t),y(t))$ remains in
the rectangle $\rectangle=[0,\xd]\times[0,k_2+\xd]$.
But if, furthermore, $\xd$ is small enough such that
$\rectangle$ lies entirely below the parabola $y=U(x)$,
then, when $(x(t),y(t))\in\rectangle$, we have $\dot{x}(t)>0$, which
entails that $x(t)$ is eventually greater than $\xd$, a
contradiction.
This shows that, for $\xd>0$ small enough, we have necessarily
$$\limsup_{t\rightarrow\infty}x(t)\geq \xd.$$

%%%
Let us now calculate the largest value of $\xd$ such that
$(x,y)\in\rectangle$ implies $y<U(x)$, that is, the largest $\xd$ such
that
$$\min_{x\in[0,\xd]}U(x)\geq k_2+\xd.$$
From the concavity of $U$, the minimum of $U$ on the interval $[0,\xd]$
is attained at $0$ or $\xd$.
Thus the optimal value of $\xd$ is the minimum of
$U(0)-k_2=\frac{k_1}{a}-k_2$
and the positive solution to $U(x)-k_2=x$, which
is
$$\frac{1-k_1-a+\sqrt{(1-k_1-a)^2+4(k_1-ak_2)}}{2}.$$
This proves \eqref{eq:m0limsup}.

%{\color{red}What about $k_1=ak_2$?}

        %%%%%%%%% (iii-special) /persistence/extinction
        \medskip
        \noindent\eqref{item:Am0-iii-special}
        %% cas  $k_1=ak_2$.
        Assume that $k_1=ak_2$.
        With the change of variable $\yd=y-k_2$, the system
        \eqref{eq:simple} becomes
        \begin{equation*}
\left\{\begin{aligned}
          \dot{x}% =&\,x(1-x)-\frac{ax(\yd+k_2)}{x+k_1}\\
                 %=&\,\frac{x}{x+k_1}\Bigl(-(x-1)(x+k_1)-a\yd-ak_2\Bigr)\\
                 %=&\,\frac{x}{x+k_1}\Bigl((1-k_1)x-a\yd-x^2\Bigr),\\
                 =&\,\frac{ax}{k_1+x}\Bigl(V(x)-\yd \Bigr),\\
%\intertext{and}
          \dot{\yd}=&\,b\,\frac{\yd+k_2}{x+k_2}(x-\yd),
\end{aligned}\right.
        \end{equation*}
where $V(x)=\frac{1}{a}\bigl((1-k_1)x-x^2\bigr)$.
         The second equation shows that $\dot{\yd}>0$ when $\yd<x$, and
        $\dot{\yd}<0$ when $\yd>x$.
         The first equation shows that $\dot{x}>0$ when $(x,\yd)$ is
         above the parabola $\yd=V(x)$, and $\dot{x}<0$ when $(x,\yd)$ is
         below the parabola $\yd=V(x)$.
         % Since the polynomial $V$ can
         % be factorized by $x$, we see that the convergence of $x$ is
         % very slow, whereas the convergence of $\yd-x$ to $0$ is
         % exponential.

        $\bullet$ Assume that $1-k_1-a>0$, that is,
        $V'(0)=(1-k_1)/a> 1$. Then, the parabola
        $\yd=V(x)$ is above the line $\yd=x$ for all $x$ in the
        interval $]0,\xd[$, where $\xd$ is the non-zero solution to
        $V(x)=x$, that is,
        $$\xd=1-k_1-a.$$
        Let us show that $\limsup x(t)\geq\xd$. Assume the
        contrary, that is, $\limsup x(t)<\delta$ for some
        $\delta<\xd$.
        For $t$ large enough, say, $t\geq t_\delta$,
        we have $x(t)<\delta$.
        Let us first prove that $\abs{\yd(t)}<\delta$
        for $t$ large enough.
        If $\yd(t_\delta)<\delta$, we have, for all $t\geq t_\delta$, as long as
         $\yd(t)<\delta$,
        \begin{equation*}
          \dot{\yd}(t)<b\,\frac{\xd+k_2}{k_2}(\delta-\yd(t)).
        \end{equation*}
        Since the constant function $\yd=\delta$ is a solution to
        $\dot{\yd}=b\,\frac{\xd+k_2}{k_2}(\delta-\yd)$, we deduce that
        $\yd(t)$ remains in $[-k_2,\delta]$ for all $t\geq t_\delta$.
        Furthermore, if $\yd(t)<-\delta$, for $t\geq t_\delta$, we
        have $\dot{\yd}(t)>0$, thus
       \begin{equation*}
          \dot{\yd}(t)>b\,\frac{\yd(t_\delta)+k_2}{k_2+\delta}(-\yd(t)).
        \end{equation*}
        Thus
        \begin{equation*}
          \yd(t)\geq
          y(t_\delta)
            \exp\CCO{-b\,\frac{\yd(t_\delta)+k_2}{k_2+\delta}\,(t-t_\delta)},
        \end{equation*}
        which proves that $\yd(t)$ enters $]-\delta,\delta[$ in finite
        time.
        Similarly, if $\yd(t_\delta)>\delta$, then, for all $t\geq t_\delta$
        such that $\yd(s)>\delta$ for all $s\in[t_\delta,t]$, we have
        \begin{equation*}
          \dot{\yd}(t)<b\frac{\yd(t_\delta)+k_2}{k_2}(\delta-\yd(t)),
        \end{equation*}
        thus
        \begin{equation*}
          \yd(t)<
          \delta+(\yd(t_\delta)-\delta)
              \exp\CCO{-b\frac{\yd(t_\delta)+k_2}{k_2}(t-t_\delta)},
        \end{equation*}
        which proves that $\yd(t)<\delta$ after a finite time.

        %%%%%
        We have proved that,
        for $t$ large enough, $(x(t),\yd(t))$ stays in the box
        $[0,\delta[\times]-\delta,\delta[$. Since $V(x)>x$ for all
        $x\in]0,\xd[$, we deduce that, for $t$ large enough,
        we have
        \begin{equation*}
          \dot{x}(t)>x(t)\frac{V(\delta)-\delta}{k_1+\delta},
        \end{equation*}
        which shows that $x(t)>\delta$ for $t$ large enough, a
        contradiction. This proves \eqref{eq:limsup-ak2=k1}.

        $\bullet$ Assume that $1-k_1-a\leq 0$, that is,
        $V'(0)=(1-k_1)/a\leq 1$.
        Then,
        the portion of the parabola $\yd=V(x)$ which lies
        in $]0,+\infty[\times]-k_2,+\infty[$,
        is below the line $\yd=x$.
        This means that, for any $\epsilon>0$ such that
        $k_2-\epsilon>0$, the system \eqref{eq:simple}
        has no other equilibrium point than $E_2$ in the invariant
        attracting compact set $[0,1]\times[k_2-\epsilon,L]$. Since
        there cannot be any periodic orbit around $E_2$ (because $E_2$
        is on the boundary of $[0,1]\times[k_2-\epsilon,L]$),
        this entails that $E_2$ is
        attracting for all inital conditions in
        $[0,1]\times[k_2-\epsilon,L]$,
        thus for all inital conditions in $]0,+\infty[\times]0,+\infty[$.

%%%%%%%%%%%%%%%%%%%%%%%%%%%%%%%%%%%%%%%%%%%%%%%%%%%%%
        %%%%%%%%% (iii) extinction
        \medskip
        \noindent\eqref{item:Am0-iii}
        %% cas  $k_1<ak_2$
        If $k_1<ak_2$, we can use exactly the same arguments as in the case
        when $k_1=ak_2$ with $1-k_1-a\leq 0$.
\end{proof}

	%\newpage
	%%%%%%%%%%%%%%%%%%%%%%%%%%%%%%%%%%%%%%%%%%%%%%
	\subsection{Local study of equilibrium points}
	\subsubsection{Trivial critical points}
	\label{sect:localstability}
	The right hand side of \eqref{eq:simple} has continuous partial
	derivatives in the first quadrant $\R_+\times\R_+$, except on the line
	$x=m$ if $m>0$.
	The
	Jacobian matrix of the right hand side of \eqref{eq:simple}
	(for $x\not=m$ if $m>0$), is
	\begin{equation}
		\label{eq:jacobian}
		\jcb{x,y}=
		\begin{pmatrix}
			1-2x-\frac{ayk_1}{(k_1+(x-m)_+)^2}\un{{x\geq m}}
			&\frac{-a(x-m)_+}{k_1+(x-m)_+}\\
			\frac{b y^2}{(k_2+(x-m)_+)^2}\un{{x\geq m}}  & b -\frac{2b y}{k_2+(x-m)_+}
		\end{pmatrix},
	\end{equation}
	where $\un{{x\geq m}}=1$ if $x\geq m$ and $\un{{x\geq m}}=0$ if
	$x<m$.

	%\paragraph{Counting and localizing equilibrium points}
	We start with a result on the obvious critical points of \eqref{eq:simple} which lie on the axes.
	%%%%%%%%%%%%%%%%%%%%%%%%%%%%%%%%
	\begin{proposition}\label{prop:trivial}
		The system \eqref{eq:simple} has three trivial critical points on the
		axes:
		\begin{itemize}
			\item $E_0=(0,0)$, which is an hyperbolic unstable node,
			\item $E_1=(1,0)$, which is an hyperbolic saddle point
                        whose
			 stable manifold is the $x$ axis,
                        and with an unstable manifold which is tangent to the line
                        $(b+1)(x-1)+\frac{a(1-m)}{k_1+1-m}y=0$,
			% and
			\item $E_2=(0,k_2)$, which is
                             \begin{itemize}
                             \item an hyperbolic saddle point whose
			stable manifold is the $y$ axis,
                        with an unstable manifold which is tangent to the
                        line
                        $bx+\Bigl(b+1-\frac{ak_2}{k_1}\un{{m=0}}\Bigr)(y-k_2)=0$
                        if $m>0$
                        or if $ak_2<k_1$, where $\un{{m=0}}=1$ if
                        $m=0$ and $\un{{m=0}}=0$ otherwise,

                             \item an hyperbolic stable node
                        if $m=0$ with $ak_2>k_1$,

                             \item a semi-hyperbolic point if
                               $m=0$ and $ak_2=k_1$, which is

                                \begin{itemize}
                               \item an attracting topological node
                                 if $1-k_1-a\leq 0$, %$(1-k_1)/a\leq
                                   %1$,

                               \item a topological saddle point  if
                               $1-k_1-a> 0$. %$(1-k_1)/a> 1$.
                               In this case, the $y$ axis is the
			       stable manifold, and there is a center
                               manifold which is tangent to the line $y-k_2=x$.
                                \end{itemize}
			% saddle point whose
			% stable manifold is the $y$ axis
                             \end{itemize}
                        (Compare with
                        the case \eqref{item:Am0}
                        of \Cref{theo:Ainvariant}).

                       %{\color{red} What if $ak_2=k_1$?}
		\end{itemize}
	\end{proposition}
	%%%%%%%%%%
	\begin{proof} The nature of $E_0$, $E_1$, and $E_2$, is obvious since
	\begin{equation*}
		\jcb{0,0}=
		\begin{pmatrix}
			1 & 0\\
			0  & b
		\end{pmatrix},
		\
		\jcb{1,0}=
		\begin{pmatrix}
			-1 & \frac{-a(1-m)}{k_1+1-m}\\
			0  & b
		\end{pmatrix},
		\
		\jcb{0,k_2}=
		\begin{pmatrix}
			1-\frac{ak_2}{k_1}\un{{m=0}} & 0\\
			b  & -b
		\end{pmatrix}.
	\end{equation*}
        The results on stable and unstable manifolds of hyperbolic
        saddles are straightforward.
        In the case when $E_2$ is semi-hyperbolic,
         since it is either
        a topological node or a topological saddle
        (see \cite[Theorem 2.19]{dumortier-llibre-artes}),
        the nature of $E_2$ follows from
        Part \eqref{item:Am0-iii-special} of %Theorem
        \Cref{theo:Ainvariant}.
        In the topological saddle case, that is, when
        $m=0$ with $ak_2=k_1$ and $1-k_1-a> 0$,
        the eigen values of $\jcb{0,k_2}$ are $-b$ and $1$, with
        corresponding eigenvectors $(0,1)$ and $(1,1)$.
        Clearly, the $y$ axis is the stable manifold.
        The change of variables
        $$\XX=x,\quad \YY=(y-k_2)-x$$
        yields the normal form
              \begin{align*}
                 \dot{\XX}=&\, \dot{x}=\frac{\XX}{\XX+k_1}\Bigl((1-k_1)\XX-\XX^2-a(\XX+\YY)\Bigr)\\
                  =&\,
                    \frac{\XX}{\XX+k_1}\Bigl((1-k_1-a)\XX-\XX^2-a\YY\Bigr),\\
         \dot{\YY}=&\,
                    \dot{x}-\dot{y}=\dot{x}-b\frac{\XX+\YY+k_2}{\XX+k_2}(-\YY)
                  %=&\,\dot{\XX}
                  %  -b\frac{\XX+\YY+k_2}{\XX+k_2} \YY\\
                  = \dot{\XX}-b\CCO{1+\frac{\YY}{\XX+k_2}}\YY\\
                  =&\,-b\YY+\dot{\XX}-b\frac{\YY^2}{\XX+k_2}.
       \end{align*}
       We can thus write
       	\begin{equation}
\begin{aligned}
  \dot{\XX}=&\,A(\XX,\YY),\\
  \dot{\YY}=&\,-b\YY+B(\XX,\YY),
\end{aligned}
\end{equation}
where $A$ and $B$ are analytic and their jacobian matrix at $(0,0)$ is
$0$.
In the neighborhood of $(0,0)$, the equation
$0=-\YY b +B(\XX,\YY)$ has the unique solution $\YY=\Cf(\XX)$, where
$$\Cf(X)=\frac{k_2a}{bk_2}X+O(X),$$
and $\Cg(\XX)=A(\XX,\Cf(\XX))$ has the form
\begin{equation*}
\Cg(\XX)=\frac{X^2}{k_2}\CCO{1+k_1-a-\frac{a^2k_2}{bk_1}}+O(X).
\end{equation*}
From \cite[Theorem 2.19]{dumortier-llibre-artes}, we deduce
       that there exists an unstable center manifold which is
       infinitely tangent to the line $\YY=0$.
	\end{proof}
	\subsubsection{Counting and localizing equilibrium points}
	
	Let us now look for  critical points outside the axes, i.e.,
	critical points $\EE=(\xx,\yy)$ with $\xx> 0$ and $\yy>0$.
	From the results of \Cref{subsec:persistence}, such points
	are necessarily in $\invrg$, in particular they satisfy $\xx\geq m$.
	We have, obviously:	
	%%%%%%%%%%%%%%%%%%%%
	\begin{lemma}\label{lem:critiquesdansA}
		The set of equilibrium points of \eqref{eq:simple}
		which lie in the open quadrant $]0,+\infty[\times ]0,+\infty[ $
		consists
		of the intersection points of the curves
		\begin{align}
			\xx(1-\xx)\CCO{k_1+\xx-m}&=a\CCO{{k_2+\xx-m}}(\xx-m) \label{eq:xstar2},\\
			{k_2+\xx-m}&=\yy.\label{eq:ystar2}
		\end{align}
		Furthermore, these points lie in $\invrg$.
	\end{lemma}
	%%%%%%%%%%%%%%%%%%%	
	We shall see that, when $m>0$,
	the system \eqref{eq:simple} has always at least
	one equilibrium point
	in $]0,+\infty[\times ]0,+\infty[ $,
	whereas, for $m=0$, some condition is necessary for the existence of
	such a point.

	\bigskip
	%%%%%%%%%%%%%%%%%%%%%%%%%%%%%%%%% CAS m>0 %%%%%%%%%%%%%%%%%%%%%%%%%%%%%%
	$\bullet$ \underline{When $m>0$}, the solutions to \eqref{eq:xstar2}
	lie at the abscissa of the intersection of the parabola
	$z=\Pf(\xx):=a\CCO{{k_2+\xx-m}}(\xx-m)$ and of the third degree curve
	$z=\Pg(\xx):=\xx(1-\xx)\CCO{k_1+\xx-m}$. %see Figure \cref{fig:m>0}.
	We have $\Pf(m)-\Pg(m)=-\Pg(m)=-k_1m(1-m)<0$ and,
	for $x>1$, we have $\Pf(x)<0$ and $\Pg(x)>0$, thus $\Pf(x)-\Pg(x)>0$.
	This implies that
	the curves of $\Pf$ and $\Pg$ have at least one
	intersection whose abscissa is
	greater than $m$, and that the abscissa of any such intersection
	lies necessarily in the interval $]m,1[$.
	%%%%%%%%%%
	The change of variable $X=x-m$ leads to
	\begin{equation}
		\label{eq:R(x)}
		R(X):=\Pf(x)-\Pg(x)
		=X^3+\alpha_2 X^2+\alpha_1 X+\alpha_0,
	\end{equation}
	with
	\begin{equation}
		\label{eq:Rcoeffs}
		\alpha_2=a+k_1-1+2m,\
		\alpha_1%=ak_2+k_1(2m-1)-m(1-m),\
		=m^2+m(2k_1-1)+ak_2-k_1,\
		\alpha_0=-k_1m(1-m).
	\end{equation}
	By Routh's scheme (see \cite{gantmacher}),
	the number $\nroot$ of roots of
	\eqref{eq:R(x)} with positive real part,
	counted with multiplicities,
	is equal to the number of changes of sign of the sequence
	\begin{equation}
		\label{eq:Routhseq}
		V:=\CCO{1,\alpha_2,\alpha_1-\frac{\alpha_0}{\alpha_2},\alpha_0},
	\end{equation}
	provided that all terms of $V$ are non zero.
	Thus $\nroot=3$ when
	\begin{equation}
		\label{eq:n=3}
		\alpha_2<0 \text{ and } \alpha_1\alpha_2<\alpha_0,
	\end{equation}
	and, in all other cases,  $\nroot=1$.
	When $\nroot=1$, we know that the number $\nrootp$ of real positive roots of
	$R$ %\eqref{eq:R(x)}
	is exactly 1.
	When $\nroot=3$, we have either $\nrootp=1$ if $R$ has two complex
	conjugate roots, or $\nrootp=3$. So, we need to examine
	when all roots of $R$ %\eqref{eq:R(x)}
	are real numbers.
	A very simple method to do that for cubic polynomials is described by
	Tong \cite{tong04}:
	%note={cited in McNamee and Pan, Numerical Methods for Roots
        %of Polynomials, Part II}
	a necessary and sufficient condition for $R$ to have three distinct
	real roots is that $R$ has a local maximum and a local minimum, and
	that these extrema have opposite signs. The abscissa of these extrema
	are the roots of the
	derivative $R'(X)=3X^2+2\alpha_2X+\alpha_1$, thus $R$ has three
	distinct real roots if, and only if, the following conditions are
	simultaneously satisfied:
	\begin{enumerate}[(i)]
		\item\label{cond:discrpos} The discriminant $\discrimRp$
		of $R'$ is positive,
		%i.e., $\alpha_2^2-3\alpha_1>0$,
		
		\item\label{cond:RR} $R(\xpmin)R(\xpmax)<0$, where
                  $\xpmin$ and $\xpmax$ are the
		distinct roots of $R'$.
	\end{enumerate}
	If  $R(\xpmin)R(\xpmax)=0$ with $\Delta_{R'}>0$,
	the polynomial $R$ still has three real roots,
	two of which coincide and differ from the third one.
	If $R(\xpmin)R(\xpmax)=0$ with $\Delta_{R'}=0$,
	it has
	a real root with multiplicity 3, which is $\xpmin=\xpmax$,
	and if $\Delta_{R'}=0$ with  $R(\xpmin)R(\xpmax)\not=0$, it has only
	one real root.
	Fortunately, all radicals disappear in the calculation of $R(\xpmin)R(\xpmax)$:
	\begin{equation*}
		R(\xpmin)R(\xpmax)
		=\frac{1}{27}\CCO{\RR}.
		%MAPLE : (4/27)*a2^3*a0-(1/27)*a2^2*a1^2+a0^2-(2/3)*a2*a1*a0+(4/27)*a1^3
	\end{equation*}
	In particular, Conditions \eqref{cond:discrpos} and \eqref{cond:RR} can be summarized as
	\begin{equation}
		\label{eq:tong2004}
		%\begin{gathered}
		\alpha_2^2-3\alpha_1>0\text{ and }
		\RR<0.
		%\end{gathered}
	\end{equation}

	Let us now examine what happens when one term of the sequence $V$ in
	\eqref{eq:Routhseq} is zero.
	We skip temporarily the case $\alpha_0=0$, which is equivalent to $m=0$.
	\begin{itemize}
		\item If $\alpha_2\alpha_1=\alpha_0$, we have
		$$R(X)=(X+\alpha_2)(X^2+\alpha_1),$$
		and $\alpha_2$ and $\alpha_1$ have opposite signs, because
		$\alpha_0<0$.
		Thus, in that case,  $R$ has a unique positive root, which is
		$\sqrt{-\alpha_1}$ if $\alpha_2>0$, and $-\alpha_2$ if $\alpha_2<0$.

		\item
		If $\alpha_2=0$, the derivative of $R$ becomes
		$R'(X)=3X^2+\alpha_1$.
		If $\alpha_1>0$, $R$ is increasing on $]-\infty,\infty[$, thus it has
		only one (necessarily positive) real root.
		If $\alpha_1=0$, we have $R(X)=X^3+\alpha_0$, thus $R$
                has only one real
		root, which is $\sqrt[\leftroot{1}\uproot{1}3]{-\alpha_0}>0$.
		If $\alpha_1<0$, $R$ is decreasing in the interval
		$[-\sqrt{-\alpha_1},\sqrt{-\alpha_1}]$,
		and increasing in $[\sqrt{-\alpha_1},+\infty[$.  Since $R(0)<0$, $R$
		has only one positive root.
		Thus, in that case too, $R$ has a unique positive root.
	\end{itemize}

From the preceding discussion, we deduce the following theorem:	
	%%%%%%%%%%%%%%%%%%%%%%% prop m>0 %%%%%%%
	\begin{theorem}\label{theo:nbdepoints}
		Assume that $m>0$.
		With the notations of \eqref{eq:Rcoeffs}, the number
                $\nrootp$ of distinct equilibrium
		points of the system \eqref{eq:simple} which lie in the open quadrant
		$]0,+\infty[\times ]0,+\infty[$ is
		\begin{enumerate}[(a)]

			\item \label{cond:n=3}
			$\nrootp=3$ if $\bigl\lgroup \alpha_2<0$,
			$\alpha_1\alpha_2<\alpha_0$,
			$\alpha_2^2-3\alpha_1> 0$, and
			$\RR < 0\bigr\rgroup,$
			
			\item \label{cond:n=2}
                        $\nrootp=2$ if $\bigl\lgroup
                        \alpha_2<0$, $\alpha_1\alpha_2<\alpha_0$, $\alpha_2^2-3\alpha_1> 0$ and
                        $\RR
			= 0\bigr\rgroup,$

			\item \label{cond:n=1}
			$\nrootp=1$ in all other cases, i.e., if $\bigl\lgroup
                        \alpha_2\geq0$ or $\alpha_1\alpha_2\geq\alpha_0$ or $\alpha_2^2-3\alpha_1\leq 0$ or
                        $\RR
			> 0\bigr\rgroup.$
			
		\end{enumerate}
	
	\end{theorem}

	%%%%%%%%%%%%%%%%%%%%%%%%%%%%%%%
	\begin{remark}\label{rem:nonempty}
		Numerical computations show that
		all cases considered in \Cref{theo:nbdepoints} are nonempty.
		See %Figure
                \Cref{fig:3points-cycle} for an example
		of positive numbers $(a,k_1,k_2,m)$ satisfying
		\eqref{eq:n=3} and \eqref{eq:tong2004}.
	\end{remark}
	%%%%%%%%%%%
	
	\bigskip
	%%%%%%%%%%%%%%%%%%%%%%%%%%%%%%% cas m=0 %%%%%%%%%%%%%%%%%%%%%%%%%%%%
	$\bullet$ \underline{When $m=0$}, the system \eqref{eq:simple} is exactly the
	system studied by
	M.A.~Aziz-Alaoui and M.~Daher-Okiye
	\cite{Daher-Aziz03,Daher-Aziz03bologn}.
	As $\xx$ is assumed to be positive, \eqref{eq:xstar2} is
	equivalent to the quadratic equation
	\begin{equation}\label{eq:xstar0}
		(1-\xx)\CCO{k_1+\xx}=a\CCO{{k_2+\xx}},
	\end{equation}
        which can be written
	\begin{equation*}
		\xx^2+\alpha_2\xx+\alpha_1=0,
	\end{equation*}
	where $\alpha_2=a+k_1-1$ and $\alpha_1=ak_2-k_1$ as in \eqref{eq:Rcoeffs}.
	The associated discriminant is
	\begin{equation}
		\label{eq:discrim2}
		\discrimQ=\alpha_2^2-4\alpha_1=(a+k_1-1)^2-4ak_2+4k_1,
	\end{equation}
	thus a sufficient and necessary condition for the existence of
	solutions to \eqref{eq:xstar0} in $\R$ is $\discrimQ\geq 0$, i.e., $k_2$
	must not be too large:
	\begin{equation}\label{eq:discriminant}
		4ak_2\leq (1-k_1-a)^2+4k_1.
	\end{equation}
	Since the sum of the solutions to \eqref{eq:xstar0}
	is $-\alpha_2$ and their product is
	$\alpha_1$, we deduce the following result:

%%%%%%%%%%%%%%%%%%%%%%%%%%%%%%%%%
	\begin{theorem}\label{prop:m=0nbdepoints}
		Assume that $m=0$.
		With the notations of \eqref{eq:Rcoeffs}, the number $\nrootp$ of
		distinct equilibrium points of the system \eqref{eq:simple}
		which lie in the open quadrant
		$]0,+\infty[\times ]0,+\infty[$ is
		\begin{enumerate}[(a)]
			\item\label{item:m0-2sol}
			$\nrootp=2$
			if
			$\discrimQ> 0$ and $\alpha_1>0$ and $\alpha_2<0$, i.e., if
			\begin{equation}
				\label{eq:2positivesol}
				4ak_2< (1-k_1-a)^2+4k_1
				\text{ and }
				ak_2>k_1\text{ and } 1-k_1-a> 0.
			\end{equation}
			% \eqref{eq:discriminant} is satisfied with
			% $ak_2>k_1$ and $1-k_1-a> 0$.

			\item\label{item:m0-1sol}
			$\nrootp=1$ if
			$\bigl\lgroup \discrimQ > 0$ and $\bigr\lgroup\alpha_1<0$ or $(\alpha_1=0$ and
			$\alpha_2<0)\bigr\rgroup\bigr\rgroup$, or $\bigl\lgroup \discrimQ=0 $ and $\alpha_2<0 \bigr\rgroup$
			i.e., if
			\begin{equation*}
			\bigl\lgroup\bigl\lgroup
			4ak_2< (1-k_1-a)^2+4k_1
			\bigr\rgroup
			\text{ and }
			\bigl\lgroup
			ak_2< k_1\text{ or }
			\bigl(ak_2=k_1\text{ and } 1-k_1-a> 0\bigr)
			\bigr\rgroup\bigr\rgroup, \\
			\end{equation*}
			or $\bigl\lgroup 4ak_2= (1-k_1-a)^2+4k_1$ and $1-k_1-a> 0 \bigr\rgroup$,\\
			% \eqref{eq:discriminant} is satisfied, with
			% $ak_2<k_1$
			%       or ($ak_2=k_1$ and $1-k_1-a>0$),

\item\label{item:m0-nosol} %no positive solution
			$\nrootp=0$
			if $\discrimQ<0$, or if
			$\bigl\lgroup
			\alpha_1\geq0\text{ and }\alpha_2\geq 0
			\bigr\rgroup,$ i.e., if
			$$
			\bigl\lgroup
			4ak_2> (1-k_1-a)^2+4k_1
			\bigr\rgroup
			\text{ or }
			\bigl\lgroup
			ak_2\geq k_1\text{ and } 1-k_1-a\leq 0
			\bigr\rgroup.
			$$
			% \eqref{eq:discriminant} is not satisfied, or if
			% or if ($ak_2\geq k_1$ and $1-k_1-a\leq 0$),

		\end{enumerate}
	\end{theorem}
	%%%%%%%%%%
	
%%%%%%%%%%%%%%
\begin{remark}\label{rem:+precis}
        If $m=0$ and $\nrootp=0$, the point $E_2$ is the only
        equilibrium point in the compact invariant attracting set
        $[0,1]\times[k_2-\epsilon,L]$, for any $\epsilon>0$ such
        that $k_2-\epsilon>0$, thus $E_2$ is globally attractive,
        because there is no cycle around $E_2$ (since $E_2$ is on the
        boundary of $[0,1]\times[k_2-\epsilon,L]$).
        This gives a more general condition of global attractivity of $E_2$
        than the result given in Parts \eqref{item:Am0-iii-special} and
        \eqref{item:Am0-iii} of \Cref{theo:Ainvariant}.
   % %%%%%%%%%%% COMMENTAIRE A CONSERVER
   %      To see that this result is more general, assume that
   %      $k_1+1<ak_2$, and let us prove that
   %      \begin{gather}
   %      		\bigl\lgroup
   %      		4ak_2> (1-k_1-a)^2+4k_1
   %      		\bigr\rgroup\tag{a}\\
   %      		\intertext{or}
   %      		\bigl\lgroup
   %      		ak_2\geq k_1\text{ and } 1-k_1-a< 0
   %      		\bigr\rgroup\tag{b}.
   %     \end{gather}
   %     If $1-k_1-a<0$, the condition (b) is %\eqref{item:m0-nosol} is
   %      satisfied. Assume that $1-k_1-a\geq 0$. Then, since $k_1$ and
   %      $a$ are positive, we have $(1-k_1-a)^2\leq 1$, thus (a) is
   %      satisfied.
   % %%%%%%%%%%%%%%%%%%%%%%%%%%%%%%%%%%%%%%
\end{remark}
	
	%%%%%%%%%%%%%%
	\begin{remark}\label{rem:limitcase}
		Since the roots of the polynomial $R$
		defined by \eqref{eq:R(x)}
		depend continuously
		on its coefficients, %(see \cite{Curgus06}),
		\Cref{prop:m=0nbdepoints}
		expresses the limiting localization of the equilibrium points of
		\eqref{eq:simple} when $m$ goes to 0.
		In particular,
		the case \eqref{item:m0-2sol} of \Cref{prop:m=0nbdepoints}
		is the limiting case of \eqref{cond:n=3} in
		\Cref{theo:nbdepoints}.
		Indeed, it is easy to check that
		Condition \eqref{eq:2positivesol}, with $m=0$, is a
                limit case of
		\eqref{eq:n=3} and \eqref{eq:tong2004}.
		This means that,
		in the case \eqref{cond:n=3} of
		\Cref{theo:nbdepoints},
		when $m$ goes to 0,
		one of the equilibrium points in the open quadrant
		$]0,+\infty[\times ]0,+\infty[$ goes to $E_2$ and leaves the open
		quadrant $]0,+\infty[\times ]0,+\infty[$.
		(Note that, when $m=0$, the equilibrium point $E_2=(0,k_2)$ is in $\invrg$.)
		%lies in the intersection of $\invrg$ and the line defined by \eqref{eq:ystar2}.
	\end{remark}
	%%%%%%%%%

	%%%%%%%%%%%%%%
	\begin{remark}
		When $k_1=k_2:=k$, since $\xx>m$, \Cref{eq:xstar2} is
                equivalent to
		$\xx(1-\xx)=a(\xx-m)$, i.e.,
		$$\xx^2+\xx(a-1)-am,$$
		thus it has at most one positive solution.
		In that case,
		the coordinates of the unique non trivial equilibrium point $\Ec$
		can be explicited in a simple way, and we have
		$$\Ec=\CCO{
			\frac{ 1-a+\sqrt{(1-a)^2+4am} }{2},k+\xc-m
		}.
		$$
		If $a\geq 1$, the point $\Ec$ converges to $E_2$ when $m$ goes to
		0. If $a>1$, it converges to $(1-a,1-a+k)$.
	\end{remark}

	%%%%%%%%%%%%%%%%%%%%%%%%%%%
	\subsubsection{Local stability}
	
	Let $\Ec=(\xc,\yc)$ be an equilibrium point of \eqref{eq:simple} in
	the open quadrant  $]0,+\infty[\times ]0,+\infty[ $.
	Since $\Ec$ is necessarily in $\invrg$, we get,
	using \eqref{eq:jacobian} and \eqref{eq:ystar2},
	\begin{equation}\label{eq:jacobian-invrg}
		\jcb{\xc,\yc}=
		\begin{pmatrix}
			1-2\xc-\frac{a\yc k_1}{(k_1+\xc-m)^2} &
                                                  \frac{-a(\xc-m)}{k_1+\xc-m}\\
			b  & -b
		\end{pmatrix}.
	\end{equation}
	The characteristic polynomial of $\jcb{\xc,\yc}$ is
	\begin{equation*}
		\charac(\lambda)
		=\lambda^2+s\lambda+p,
	\end{equation*}
	where
	\begin{align}
		s&=-\trace\CCO{\jcb{\xc,\yc}}
		=-1+2\xc+\frac{a\yc k_1}{(k_1+\xc-m)^2}+b,
		      \label{eq:s}\\
		p&=\det \CCO{\jcb{\xc,\yc}}
                  =b\CCO{-1+2\xc+\frac{a\yc
                    k_1}{(k_1+\xc-m)^2}+\frac{a(\xc-m)}{k_1+\xc-m}}.
                     \label{eq:p}
	\end{align}
	The roots of $\charac$ are real if, and only if, $\discrimC\geq0$, where
	\begin{equation*}
		\discrimC=s^2-4p=\CCO{ -1+2\xc+\frac{a\yc k_1}{(k_1+\xc-m)^2} -b}^2
		-4b\frac{a(\xc-m)}{k_1+\xc-m}.
	\end{equation*}
        The point $\Ec$ is non-hyperbolic if one of the roots of
        $\charac$ is zero (that is, if $p=0$), or if
        $\charac$ has two conjugate purely imaginary roots (that is,
        if $s=0$ with $p>0$).
       If only one root of $\charac$ is zero, that is, if $p=0$ with
        $s\not=0$, the point $\Ec$ is semi-hyperbolic.
%We will discuss later the non-elementary case.
\paragraph{a- Hyperbolic equilibria}

      When $\Ec$ is hyperbolic, we get,
	using the Routh-Hurwitz criterion, that $\Ec$ is
	% nombre de changements de signe dans la suite $(1,s,p)$, qui est la
	% 1e colonne du tableau de Routh-Hurwitz
	\begin{itemize}
		\item a saddle point if $p<0$,
		\item an unstable node if $s<0$ and $p>0$ with $\discrimC>0$,
		\item an unstable focus if $s<0$ and $p>0$ with $\discrimC<0$,
		\item an unstable degenerated node if $s<0$ and $p>0$ with $\discrimC=0$,
		\item a stable node if $s>0$ and $p>0$ with $\discrimC>0$,
		\item a stable degenerated node if $s>0$ and $p>0$ with $\discrimC=0$,
		\item a stable focus if $s>0$ and $p>0$ with $\discrimC<0$.

	\end{itemize}
 \begin{remark}\label{rem:mgeq1/2}
	An obvious sufficient condition for any equilibrium point
	$\Ec\in\invrg$
	to be stable hyperbolic is $m\geq 1/2$,
	since $\xc>m$. This condition can be slightly improved, as we shall see
        in the study of global stability (see \Cref{theo:globstable}).
\end{remark}

%%%%%%%%%%%%%%%%%%%%%%%%%%%%  POINCARE-HOPF %%%%%%%%%%%%%%%%%%%%%%%%%%%%%%%%%%
\paragraph{Application of the Poincar\'e index theorem}
When $\Ec$ is an hyperbolic equilibrium,
%elementary equilibrium point (that is, hyperbolic, or semi-hyperbolic)
 its index is either $1$ (if
it is a node or focus) or $-1$ (if it is a saddle). Let $\nrootd$ be the number of distinct equilibrium points, which we denote by $E_1^*,...,E_{\nrootd}^*$, and let $I_1,...,I_{\nrootd}$ their respective indices.
As we shall see in the proof of the next theorem, by a generalized version of the Poincar\'e index theorem, we have $I_1+...+I_{\nrootd}=1$.
When all equilibrium points are hyperbolic, this allows us to count the number of nodes or foci and of saddles.

	%%%%%%%%%%%% POINCARE-HOPF INDEX THEOREM
\begin{theorem}\label{theo:poincare-hopf}
%{\bf(application of the Poincar\'e index theorem)}
Assume that all equilibrium points of the system \eqref{eq:simple} which lie in
                the open quadrant $]0,+\infty[\times ]0,+\infty[$
		(equivalently, in the interior of $\invrg$) are
                hyperbolic,
                and let $\nrootp$ be their number.
		\begin{enumerate}
			\item Assume that $m>0$. Then $\nrootp$ is
                          equal to 3 or 1.
			\begin{itemize}

				\item If $\nrootp=1$, the unique
                                  equilibrium point in the interior of
				$\invrg$ is a node or a
				focus.

                        \item If $\nrootp=3$, the system
                                  \eqref{eq:simple} has one saddle
				point and two nodes or foci in the
                                interior of $\invrg$.
			\end{itemize}

			\item Assume now that $m=0$. Then $\nrootp$ is
                          equal to 2, 1, or 0.
                        \begin{itemize}
				\item If $\nrootp=2$, one equilibrium
                                  point is a node or focus, and the
                                  other is a saddle.

                                \item If $\nrootp=1$, the unique
                                  equilibrium point in the interior of
				$\invrg$ is a node or a focus.

			\end{itemize}

		\end{enumerate}
	\end{theorem}
	%%%%%%%%%
% \begin{remark}
% %If $\nrootp=2$, then necessarily one critical point is nonhyperbolic.
% \end{remark}
	\begin{proof} %{\Cref{theo:poincare-hopf}}
        Let $N$ (respectively $S$) denote the number of nodes or foci
        (respectively of saddles) among the hyperbolic singular points
        which lie in $\invrg$.
        % Let us denote by $\vectorfield$ the vector field
	% associated with \eqref{eq:simple}.

        \bigskip
	1. Assume that $m>0$.
	By \Cref{theo:Ainvariant}, the vector field
        $\vectorfield=v_1\frac{\partial}{\partial x}
                       +v_2\frac{\partial}{\partial y}$
        generated by \eqref{eq:simple}
	is directed inward along the boundary of $\invrg$.
	By continuity of $\vectorfield$, we can round the corners of $\invrg$
	and define a compact domain
	$\invrg'\subset\invrg$ with smooth boundary
	which contains all
	critical points of $\invrg$,	and such that $\vectorfield$
	is directed inward along the boundary of $\invrg'$.
	% Let $N$ denote the number of nodes or
	% foci and let $S$ denote the number of saddles in the interior
        % of $\invrg$.
	Applying a generalized version of the Poincar\'e index theorem
        (see e.g.
	\cite{llibre-villadelprat98,gottlieb86,pugh})
        to $\vectorfield$ in $\invrg'$,
	%(see \cite[Section 3.12]{perko}), %\cite{cima}
	we get $N-S=1$. Since $1\leq N+S\leq 3$, the only possibilities are
	$(N=1\text{ and }S=0)$ or $(N=2 \text{ and }S=1)$.

        \bigskip
        %%%% $m=0$, $ak_2>k_1$
	2. Assume now that $m=0$. We use the same reasoning as for $m>0$,
        but with a different domain.
        Instead of $\invrg$, we consider the domain
        $$\invrgB=[-\epsilon,1]\times[k_2-\epsilon,L]$$
        for a small $\epsilon>0$. Thus $\invrgB$ contains $E_2$.

       \medskip
       \noindent$\bullet$        With the notations of \eqref{eq:xm=0},
        if $ak_2>k_1$, we have $y>U(x)$ for $x=0$ and for all
        $y\in[k_2,L]$.
        We have
        $$v_1=\frac{ax}{k_1+x}\Bigl(U(x)-y \Bigr).$$
        By continuity of $\vectorfield$,
        we can choose $\epsilon>0$, with $\epsilon<k_1$,
        such that the inequality $y>U(x)$
        remains true on the rectangle
        $[-\epsilon,0]\times[k_2-\epsilon,L]$.
        We then have $v_1>0$ on
        the segment $\{-\epsilon\}\times[k_2-\epsilon,L]$.
        Since $v_2>0$ for $y=k_2-\epsilon$ and $v_2<0$ for $y=L$,
        the field $\vectorfield$
	is directed inward along the boundary of $\invrgB$.
        Again, by rounding the corners,
        we can modify $\invrgB$ into a a compact domain
	$\invrgB'$ with smooth boundary which contains the same
        critical points as $\invrgB$ and such that $\vectorfield$
	is directed inward along the boundary of $\invrgB'$.
        By the Poincar\'e Index Theorem, we have
	$N'-S'=1$, where $N'$ (respectively $S'$)
        is the number of nodes or foci (respectively of saddles) in
        the interior of $\invrgB'$. If we have chosen $\epsilon$ small
        enough, the singularities of $\vectorfield$ in $\invrgB'$ are
        those which are in the interior of $\invrg$, with the addition
        of the point $E_2$, which is a node by %Proposition
        \Cref{prop:trivial}. % Thus, if $N$ (respectively $S$)
        % is the number of nodes or foci (respectively, of saddles) in
        % the interior of $\invrg$, we have
        Thus $N=N'-1$ and $S=S'$ which
        entails $N-S=0$.
	Thus, taking into account \Cref{prop:m=0nbdepoints},
        we have $N=S=1$ (if $\nrootp=2$), or $N=S=0$ (if $\nrootp=0$).

       \medskip
       \noindent$\bullet$
        If $ak_2<k_1$, $E_2$ is a saddle point, thus,
        constructing $\invrgB$ and $\invrgB'$ as precedingly,
         we have now
         $S=S'-1$ and $N=N'$.
        Furthermore, the vector field $\vectorfield$ is no more
        outward directed along the whole boundary of $\invrgB'$.%   so we
        % need Morse's generalization of the Hopf-Poincar\'e index
        % theorem \cite{morse29}.
        We use Pugh's algorithm \cite{pugh} to compute $N'-S'$:
        taking $\epsilon$ small enough such that
        the vector field $\vectorfield$ does not vanish on
        $\partial \invrgB'$, we have
        \begin{equation}\label{eq:pugh}
          N'-S'=\chi(\invrgB')-\chi(\partial \invrgB')
           +\chi(R^1_-)-\chi(\partial R^1_-)
           +\chi(R^2_-)-\chi(\partial R^2_-),
        \end{equation}
        where $\chi$ denotes the Euler characteristic, $R^1_-$ is the
        part of the boundary of $\invrgB'$ where $\vectorfield$ is
        directed outward, and $R^2_-$ is the part of $\partial R^1_-$ where
        $\vectorfield$ points to the exterior of $R^1_-$.
        Since $k_2<k_1/a$, we see that
        the parabola $y=U(x)$ crosses
        the line  $\{x=-\epsilon\tq y>k_2\}$
        at some point $(-\epsilon,\tangency)$,
        so that the part of the boundary of $\invrgB$ where
        $\vectorfield$ points outward is the segment
        $\{-\epsilon\}\times[k_2-\epsilon,\min(\tangency,L)]$.
        Thus, for small $\epsilon$, % the part of the boundary of
        % $\invrgB'$ along which $\vectorfield$ is directed outward
        % is an open arc $R^1_+$ whose extremities are tangency
        % points.
        $R^1_-$ is an arc whose extremities are tangency
        points.
        Observe also that, since $v_1<0$ for $x<0$ and $v_2<0$ for
        $y>k_2+x>0$, the field $\vectorfield$
        points toward the interior of $R^1_-$
        at those tangency points, thus $R^2_-$ is empty.
        Formula \eqref{eq:pugh} becomes
       \begin{equation*}
       \Sigma(\vectorfield)=1-0
        +1-2
         +0-0=0,
       \end{equation*}
       that is,
       $N-S=N'-(S'-1)=1$. Since, by \Cref{prop:m=0nbdepoints},
%$\nrootp$
we have $N+S=1$, we deduce that $N=1$ and $S=0$.
\end{proof}

%%%%% NON HYPERBOLIC EQUILIBRIA
\paragraph{b. Semi hyperbolic equilibria}
 %% CAS semi-hyperbolique
        % $\bullet$ {\em Semi-Hyperbolic case:}
         This is when
          $p=0$ and $s\not= 0$. The set of parameters such that $p= 0$ is nonempty. Indeed, the values $a=0,5$, $b=0,01$, $m=0,001$, $k_2=0,25$, $k_1=0,08$ lead to $p=-0.1003032464$ with $\alpha_2=0.044161 >0$ and $a=0,5$, $b=0,01$, $m=0,001$, $k_2=0,25$, $k_1=0,112$ lead to $p=0.002422466814$ with $\alpha_2= 0.012225>0.$
 Since $\alpha_2$ is a linear function of $k_1$, this shows that $\alpha_2>0$ for $a=0,5$, $b=0,01$, $m=0,001$, $k_2=0,25$ and $0,08 \leq k_1 \leq 0,112.$
  Thus, by \Cref{theo:nbdepoints}, for all these values, the number $\nrootp$ of equilibrium points remains equal to $1$.
           By the intermediate value theorem, we deduce that there exists a value $k_1$, with $0,08 \leq k_1 \leq 0,112$, such that, for $a=0,5$, $b=0,01$, $m=0,001$, $k_2=0,25$, the unique equilibrium point satisfies $p=0.$

         From \eqref{eq:s},  \eqref{eq:p} and \eqref{eq:Rcoeffs}, it is obvious that we can chose $b$ such that $s \not= 0$ without changing $p=0$ nor the coefficients $\alpha_0$, $\alpha_1$, $\alpha_2$.

          For $p=0$, the
Jacobian matrix $\jcb{\xc,\yc}$ is
\begin{equation*}
		\jcb{\xc,\yc}=%\vphi'(0,0)=
		\begin{pmatrix}
			a \rho & -a \rho\\
			b  & -b
		\end{pmatrix}.
	\end{equation*}
The change of variables
$$\KX=\frac{a\rho \YY-b\XX}{a\rho-b},\quad \KY=\frac{\XX-\YY}{a\rho-b}$$
yields	
	\begin{align*}	
v_1=&a\rho (a \rho-b)\KY+a^2 \rho\frac {k_1(\yc a\rho-b\kappa)-\rho \kappa^3}{\kappa^3}\KY^2-\frac {a k_1(\kappa-\yc )+\kappa^3}{\kappa^3}\KX^2\\	
&-a\frac {k_1\kappa(b+a\rho )+\rho(2 \kappa^3-\yc k_1a)}{\kappa^3}\KY\KX+a^3k_1 \rho^2\frac {b\kappa-\yc a\rho}{\kappa^3(\kappa+\KX+\rho\,a\KY)}\KY^3\\
&-a k_1\frac { \yc-\kappa }{\kappa^3(\kappa+\KX+\rho a\KY)}\KX^3+a k_1\frac {b\kappa+2a\rho\kappa-3\yc a\rho }{\kappa^3(\kappa+\KX+\rho a\KY)}\KX^2\KY\\
&+a^2 k_1\rho \frac {2 b\kappa+a\rho\kappa-3\yc a\rho}{\kappa^3(\kappa+\KX+\rho a\KY)}\KY^2\KX,\\
v_2=&b(\rho\,a-b)\KY +b\frac {-b^2+2b\rho\,a-\rho^2a^2}{\KX+\rho\,a\KY+\yc}\KY^2.
\end{align*}
The coordinates of $\vectorfield$ are, in the basis
$(\frac{\partial}{\partial \KX},\frac{\partial}{\partial \KY})$,
\begin{align*}
  \dot{\KX}=&\,\frac{1}{a\rho-b}(a\rho \dot{\YY}-b\dot{\XX})=\frac{1}{a\rho-b}(a\rho v_2-bv_1)\\
   = &\, \frac{b}{b-\rho\,a}\Biggl\lgroup -\frac {(- k_1\yc
       a+\kappa^3+ak_1)\KX^2}{\kappa^2}+\frac {a k_1(\kappa - \yc
       )}{\kappa^3(\kappa+\KX+\rho\,a\KY)}\KX^3\\
       &+\frac{a\rho(-\kappa^3\rho a+\yc k_1a^2\rho-\kappa a k_1 b)}{\kappa^3} \KY^2
    +\frac {a^3 k_1\rho^2(\kappa b-\yc
         a\rho)}{\kappa^3(\kappa+\KX+\rho\,a\KY)}\KY^3\\
       &-\frac {a\kappa(k_1 b+a k_1\rho\,+2\kappa^2\rho\,)-2\yc k_1a^2\rho\,}{\kappa^3}\KY\KX-\rho\,a \frac {-b^2+2b\rho\,a-\rho^2a^2}{\KX+\rho\,a\KY+\yc}\KY^2\\
    &+\frac{a k_1(b\kappa +2a\kappa \rho\,-3\yc a\rho\,)}{\kappa^3(\kappa+\KX+\rho\,a\KY)}\KX^2\KY+\frac{a^2 k_1\rho(2b\kappa+a\kappa\rho-3\yc a\rho)}{\kappa^3(\kappa+\KX+\rho\,a\KY)}\KY^2\KX\Biggr\rgroup,\\
      \dot{\KY}=&\,-\frac{1}{a\rho-b}(\dot{\XX}-\dot{\YY})=\frac{1}{a\rho-b}( v_1-v_2)\\
           =&\,(a\rho-b)\KY+\frac{1}{b-\rho\,a}\Biggl\lgroup a^2\rho \frac{ k_1b\kappa+\rho \kappa^3-\yc k_1a\rho}{\kappa^3}\KY^2+\frac {a k_1\kappa- k_1\yc a+\kappa^3}{\kappa^3}\KX^2 \\
&+a\frac { k_1 b\kappa+ak_1\rho \kappa+2\rho\,\kappa^3-2\yc k_1a\rho\,}{\kappa^3}\KX\KY
+k_1a^3\rho^2\frac {\yc a\rho-b \kappa}{\kappa^3(\kappa+\KX+\rho\,a\KY)}\KY^3\\
&+a k_1\frac {\yc -\kappa}{\kappa^3 (\kappa+\KX+\rho\,a\KY)}\KX^3 
 +b \frac {(a\rho-b)^2}{\KX+\rho\,a\KY+\yc}\KY^2\\
&+a k_1\frac {3\yc a\rho\,-\kappa (b+2a\rho\,)}{\kappa^3(\kappa+\KX+\rho\,a\KY)}\KX^2\KY+ k_1 a^2 \rho \frac {3\yc a \rho-a\rho\kappa-2b \kappa}{\kappa^3 (\kappa+\KX+\rho\,a\KY)}\KY^2\KX \Biggr\rgroup .
\end{align*}
We can thus write
	\begin{equation}\label{semihyper}
\begin{aligned}
  \dot{\KX}=&\,A(\KX,\KY),\\
  \dot{\KY}=&\,\lambda \KY  +B(\KX,\KY),
\end{aligned}
\end{equation}
where $A$ and $B$ are analytic and their jacobian matrix at $(0,0)$ is
$0$ and $\lambda >0$. It is not easy to determine  $\KY = f(\KX)$  the solution to the equation $\lambda \KY +B(\KX, \KY) = 0$ in a neighborhood of the point $(0, 0)$, for that we use implicit function theorem. We find:

\textbf{Case 1:} If $\kappa^3- k \yc a+a k_1\kappa \neq 0$, we have
\begin{equation*}
f(\KX)= -\frac {\kappa^3- k \yc a+a k_1\kappa}{\kappa^3 ( b+\rho^2a^2-\rho\,ab )}\KX^2,
\end{equation*}
and $g(\KX)=A(\KX,f(\KX))$ has the form
\begin{equation*}
g(\KX)=\frac{b}{b-a\rho}(\frac{\kappa^3- k \yc a+a k_1\kappa}{\kappa^3})\KX^2.
\end{equation*}
We apply  \cite[Theorem 2.19]{dumortier-llibre-artes} to System \eqref{semihyper}. Since the power of $\KX$ in $f(\KX)$ is even,
we deduce from Part (iii) of \cite[Theorem 2.19]{dumortier-llibre-artes}:
\begin{lemma}%\label{lem:saddlenode}
If $\Ec$ is a semi-hyperbolic equilibrium of \eqref{eq:simple} in the positive quadrant
 $]0,+\infty[\times ]0,+\infty[$, and if $\kappa^3- k \yc a+a k_1\kappa \neq 0$, then $\Ec$ is a saddle-node, that is,
its phase portrait is the union of one parabolic and two hyperbolic sectors. In this case, the index of $\Ec$ is 0.
\end{lemma}

\textbf{Case 2:} if $\kappa^3- k \yc a+a k_1\kappa = 0$, we have
\begin{equation*}
f(\KX)= \frac{ak_1(\kappa -\yc)}{\kappa^4(a\rho -b)^2}\KX^3,
\end{equation*}
and $g(\KX)=A(\KX,f(\KX))$ has the form
\begin{equation*}
g(\KX)= \frac{bak_1(\kappa -\yc)}{\kappa^4(a\rho -b)^2}\KX^3.
\end{equation*}
Again, we apply  \cite[Theorem 2.19]{dumortier-llibre-artes} to System \eqref{semihyper}. Since the power of $\KX$ in $f(\KX)$ is odd, we look at the cofficient of $\KX^3$ and we have two possibilities:\\
\textbf{P1:} If $k_1>k_2$, we deduce from Part (ii) of \cite[Theorem 2.19]{dumortier-llibre-artes}:
\begin{lemma}%\label{lem:saddlenode}
If $\Ec$ is a semi-hyperbolic equilibrium of \eqref{eq:simple} in the positive quadrant
 $]0,+\infty[\times ]0,+\infty[$, and if $\kappa^3- k \yc a+a k_1\kappa = 0$ with $k_1>k_2$, then $\Ec$ is a unstable node. In this case, the index of $\Ec$ is 1.
\end{lemma}
\noindent\textbf{P2:} If $k_1<k_2$, we deduce from Part (i) of \cite[Theorem 2.19]{dumortier-llibre-artes}:
\begin{lemma}%\label{lem:saddlenode}
If $\Ec$ is a semi-hyperbolic equilibrium of \eqref{eq:simple} in the positive quadrant
 $]0,+\infty[\times ]0,+\infty[$, and if $\kappa^3- k \yc a+a k_1\kappa = 0$ with $k_1<k_2$, then $\Ec$ is a saddle. In this case, the index of $\Ec$ is -1.
 \end{lemma}
 \begin{remark}
From \Cref{theo:poincare-hopf}, when the system \eqref{eq:simple} has one equilibrium point, this point cannot be a saddle.
\end{remark}
%%%% HOPF BIFURCATION
\paragraph{Hopf bifurcation}
When $\discrimC<0$, the roots of $\charac$ are
$\frac{-s\pm i\sqrt{4p-s^2}}{2}$.
The values of $\xc$, $\yc$ and $p$ do
not depend on the parameter $b$, whereas $s$ is an affine function of
$b$, so that the eigenvalues of $\charac$ cross the imaginary axis at
speed $-1/2$ when $b$ passes through the value
$$b_0=1-2\xc+\frac{a\yc k_1}{\zc^2}.$$
Let us check the genericity condition for Hopf bifurcations.
We use the condition
 of Guckenheimer and Holmes
\cite[Formula (3.4.11)]{guckenheimer-holmes83}.
%(see also \cite{vanderbauwhede})
%\cite{wang90nonzero}
Let us denote
$$\dot{\KX}=\vfA(\KX,\KY),\quad \dot{\KY}=\vfB(\KX,\KY),$$
and $\vfA_{\KX\KY}=\frac{\partial \vfA}{\partial\KX\partial\KY}$, etc. We have
%%% GENERICITY CONDITION
\begin{align*}
 \lambda=&\vfA_{\KX\KX\KX}+\vfA_{\KX\KY\KY}+\vfB_{\KX\KX\KY}+\vfB_{\KY\KY\KY}\\
&+\frac{1}{\dzo}
\CCO{
   \vfA_{\KX\KY}(\vfA_{\KX\KX}+\vfA_{\KY\KY})-\vfB_{\KX\KY}(\vfB_{\KX\KX}+\vfB_{\KY\KY})
   -\vfA_{\KX\KX}\vfB_{\KX\KX}+\vfA_{\KY\KY}\vfB_{\KY\KY}
}\\
  =&a k_1\kappa^3 ( -2 \yc+\kappa) b_0^2\\
   & +\kappa\ ( 2 c \kappa^5+2
      k_1\yc a \kappa^3-\kappa^3c k_1 a+3\kappa c k_1\yc^2 a
      +\kappa a^2 k_1^2 \yc-2 a^2 k_1^2 \yc^2) b_0\\ 
 & - \ 2 c ( c-\yc) \kappa^6
   +a k_1 \kappa^4 c \yc+2a c\yc k_1( -2 \yc+c) \kappa^3\\
  &-3\kappa^2 k_1\yc^2 a c^2-a^2 k_1^2 \kappa c \yc^2+2 k_1^2 \yc^3 a^2c.
\end{align*}
If $\lambda < 0$, then the periodic solutions are stable limit cycles,
while if $\lambda > 0$, the periodic solutions are repelling. See %Figure
\Cref{hopf} for a numerical exemple.
	
%%%%%%%%%%%%%%%%%%%%%%%%%%%%%%%%%%%%%%%%%%%%%%%%%%%%%%%%%%%%%%%%%%%%
%\bigskip
         %% CAS NILPOTENT
   \paragraph*{c- Non-elementary equilibria} Let us rewrite the vector field
$\vectorfield=v_1\frac{\partial}{\partial x}
                       +v_2\frac{\partial}{\partial y}$
associated with \eqref{eq:simple} in the neighborhood
of an equilibrium point $\Ec=(\xc,\yc)\in\invrg$.
Let $\XX=x-\xc$ and $\YY=y-\yc$.
Since $\Ec$ is a critical point of $\vectorfield$,  we have
\begin{align*}
  v_1=&\, x(1-x)-\frac{a y(x-m)}{k_1+(x-m)}\\
     =&\, (\XX+\xc)(1-\xc-\XX)-\frac{a (\YY+\yc)(\XX+\xc-m)}{\XX+\xc+k_1-m}\\
     =&\, \xc(1-\xc)+\XX(1-2\xc-\XX)-\frac{a \yc(\xc-m)}{\XX+\xc+k_1-m}
          -\frac{a (\YY(\XX+\xc-m)+\XX\yc)}{\XX+\xc+k_1-m}\\
     =&\, \xc(1-\xc)-\frac{a \yc(\xc-m)}{\xc+k_1-m}
          +\XX(1-2\xc-\XX)\\
      &\,     +\frac{a \yc(\xc-m)}{\xc+k_1-m}
          -\frac{a \yc(\xc-m)}{\XX+\xc+k_1-m}
          -\frac{a (\YY(\XX+\xc-m)+\XX\yc)}{\XX+\xc+k_1-m}\\
     =&\,        \XX(1-2\xc-\XX)
          +\frac{a \yc(\xc-m)}{\xc+k_1-m}
          -\frac{a \yc(\xc-m)}{\XX+\xc+k_1-m}
          -\frac{a (\YY(\XX+\xc-m)+\XX\yc)}{\XX+\xc+k_1-m}\\
    =&\,        \XX(1-2\xc-\XX)
          +a \yc(\xc-m)\CCO{\frac{1}{\xc+k_1-m}
          -\frac{1}{\XX+\xc+k_1-m}}\\
     &\,     -\frac{a \Bigl(\YY(\XX+\xc-m)+\XX\yc\Bigr)}{\XX+\xc+k_1-m}\\
    =&\,\XX(1-2\xc-\XX)
          +\frac{a\yc(\xc-m)\XX}{(\xc+k_1-m)(\XX+\xc+k_1-m)}
          -\frac{a \Bigl(\YY(\XX+\xc-m)+\XX\yc\Bigr)}{\XX+\xc+k_1-m}.
\end{align*}
%%%%
For simplification, we denote
\begin{equation}%\label{eq:zc}
\zc=\xc+k_1-m,\quad \cc=\frac{\xc-m}{\xc+k_1-m},\label{eq:zc-cc}
\end{equation}
thus
\begin{equation*}
  %\label{eq:XX}
  v_1=\XX(1-2\xc-\XX)
          +\frac{a \Bigl(\XX\yc(\cc-1)-\YY(\xc-m)-\YY\XX\Bigr)}{\XX+\zc}.
\end{equation*}
Using the equality
%$$\frac{1}{x+K}=\frac{1}{K}-\frac{x}{K^2}+\frac{x^2}{K^2(x+K)},$$
\begin{equation*}
\frac{1}{x+K}=\frac{1}{K}\CCO{1-\frac{x}{K}+\dots+(-1)^n\frac{x^n}{K^n}
+(-1)^{n+1}\frac{x^{n+1}}{K^n(x+K)} },\quad
n\geq 1,
\end{equation*}
we get
\begin{align}
  v_1=&\,\XX(1-2\xc-\XX)\notag\\
     &\,+\frac{a}{\zc}
       \CCO{1-\frac{\XX}{\zc}+\frac{\XX^2}{\zc(\XX+\zc)}}
        \Bigl(\XX\yc(\cc-1)-\YY(\xc-m)-\YY\XX\Bigr)\notag\\
     =&\,\XX(1-2\xc-\XX)\notag\\
      &\,+\frac{a}{\zc}\Bigl(\XX\yc(\cc-1)-\YY(\xc-m)-\YY\XX\Bigr)\notag\\
      &\,-\frac{a\XX}{\zc^2}\Bigl(\XX\yc(\cc-1)-\YY(\xc-m)-\YY\XX\Bigr)\notag\\
      &\,+\frac{a\XX^2}{\zc^2(\XX+\zc)}
                \Bigl(\XX\yc(\cc-1)-\YY(\xc-m)-\YY\XX\Bigr)\notag\\
     =&\,\XX\CCO{1-2\xc+\frac{a}{\zc}\yc(\cc-1)}
        -\YY\frac{a}{\zc}(\xc-m)\notag\\
     &\,-\XX^2\CCO{1+\frac{a}{\zc^2}\yc(\cc-1)}
        +\XX\YY\CCO{-\frac{a}{\zc}+\frac{a}{\zc^2}(\xc-m)}\notag\\
     &\,+\frac{a\XX^2\YY}{\zc^2}
        +\frac{a\XX^2}{\zc^2(\XX+\zc)}
                \Bigl(\XX\yc(\cc-1)-\YY(\xc-m+\XX)\Bigr)\notag\\
     =&\,\XX\CCO{1-2\xc-\frac{a\yc k_1}{\zc^2}}
        -\YY a\cc
     -\XX^2\CCO{1-\frac{a\yc k_1}{\zc^3}}
        -\XX\YY\frac{ak_1}{\zc^2}\notag\\
     % &\,-\frac{a\yc k_1\XX^3}{\zc^3(\XX+\zc)}
     %    +\XX^2\YY\CCO{
     %    \frac{a}{\zc^2}
     %    -\frac{a}{\zc^2(\XX+\zc)}
     %            \Bigl(\xc-m+\XX\Bigr)}\\
     &\,-\XX^3\frac{a\yc k_1}{\zc^3(\XX+\zc)}
        +\XX^2\YY\frac{a}{\zc^2}\CCO{1-\frac{\xc-m+\XX}{\XX+\zc}}\notag\\
    =&\,\XX\CCO{1-2\xc-\frac{a\yc k_1}{\zc^2}}
        -\YY a\cc
     -\XX^2\CCO{1-\frac{a\yc k_1}{\zc^3}}
        -\XX\YY\frac{ak_1}{\zc^2}
        \label{eq:v1}\\
     &\,-\XX^3\frac{a\yc k_1}{\zc^3(\XX+\zc)}
        +\XX^2\YY\frac{ak_1}{\zc^2(\XX+\zc)}.\notag
\end{align}

Since $\yc=\xc+k_2-m$, we have also
\begin{align}
  v_2 =&\,b(\YY+\yc)\CCO{1-\frac{\YY+\yc}{k_2+x-m}}
      =b(\YY+\yc)\CCO{ 1-\frac{\YY+\yc}{\XX+\yc} }\notag\\
      =&\,b(\XX-\YY)\frac{\YY+\yc}{\XX+\yc} \notag\\%\label{eq:YY}\\
     % =&\,b(\XX-\YY)(\YY+\yc)\CCO{\frac{1}{\yc}-\frac{\XX}{{\yc}^2}
     %             + \frac{\XX^2}{ {\yc}^2(\XX+\yc) }}\\
     % =&\,b(\XX-\YY)\CCO{1-(\XX-\YY)\frac{1}{\yc}-\frac{\XX\YY}{{\yc}^2}
     %                 +\frac{\XX^2(\YY+\yc)}{ {\yc}^2(\XX+\yc) }}\\
     % =&\,b(\XX-\YY)\CCO{1-(\XX-\YY)\frac{1}{\yc}
     %                 +(\XX-\YY)\frac{\XX}{\yc(\XX+\yc)}}\\
     =&\,b(\XX-\YY)\CCO{1-(\XX-\YY)\frac{1}{\XX+\yc}}\notag\\
     % =&\,b(\XX-\YY)-\frac{b}{\yc}(\XX-\YY)^2\CCO{ 1-\frac{\XX}{{\yc}}
     %            + \frac{\XX^2}{ {\yc}(\XX+\yc) }}\\
    =&\,b(\XX-\YY)-\frac{b}{\yc}(\XX-\YY)^2\CCO{ 1-\frac{\XX}{\XX+\yc} }.\label{eq:v2}
\end{align}
This shows in particular that the linear part of $\vectorfield$ is
never zero.
Thus the only non-hyperbolic cases are the nilpotent case and the case
when $\Ec$ is a center for the linear part of $\vectorfield$.
   Let us now investigate these cases:
   \subsubsection*{$\mbox{c}_1$. Nilpotent case}

         This is when $p=0=s$. From the discussion at the beginning of Case b, it is clear that this case is nonempty.

         In this case, the Jacobian matrix $\jcb{\xc,\yc}$ is
\begin{equation*}
		\jcb{\xc,\yc}=%\vphi'(0,0)=
		\begin{pmatrix}
			b & -b\\
			b  & -b
		\end{pmatrix}.
	\end{equation*}
With the preceding notations, we thus have
\begin{align*}
       v_1=&\,b(\XX-\YY)
-\XX^2\CCO{1-\frac{a\yc k_1}{\zc^3}}
        -\XX\YY\frac{ak_1}{\zc^2}
     -\XX^3\frac{a\yc k_1}{\zc^3(\XX+\zc)}
        +\XX^2\YY\frac{ak_1}{\zc^2(\XX+\zc)}.
\end{align*}
The change of variables
$$\KX=\XX,\quad \KY=\YY-\XX$$
yields
\begin{align*}
        v_1=&\,-\KY b
         -\KX^2\CCO{1-\frac{a\yc k_1}{\zc^3}}
        -\KX(\KX+\KY)\frac{ak_1}{\zc^2}
     -\KX^3\frac{a\yc k_1}{\zc^3(\KX+\zc)}
        +\KX^2(\KX+\KY)\frac{ak_1}{\zc^2(\KX+\zc)}\\
          =&\,-\KY b
         -\KX^2\CCO{1-\frac{a\yc k_1}{\zc^3}+\frac{ak_1}{\zc^2}} %%%
         -\KX\KY \frac{ak_1}{\zc^2}\\
     &\, +\KX^3\frac{a k_1}{\zc^2(\KX+\zc)}\CCO{-\frac{\yc}{\zc}+1}
       +\KX^2\KY  \frac{ak_1^2}{\zc^2(\KX+\zc)},\\
          % =&\,-\KY b
     %     -\KX^2\CCO{1-\frac{ak_1(k_2-k_1)}{\zc^3}} %%%
     %     -\KX\KY \frac{ak_1}{\zc^2}\\
     % &\, +\KX^3\frac{a k_1}{\zc^2(\KX+\zc)}\CCO{-\frac{\yc}{\zc}+k_1}
     %   +\KX^2\KY  \frac{ak_1^2}{\zc^2(\KX+\zc)}\\
  v_2=&\,-\KY b -\KY^2\frac{b}{\yc}\CCO{ 1-\frac{\KX}{\KX+\yc} }.
\end{align*}

The coordinates of $\vectorfield$ are, in the basis
$(\frac{\partial}{\partial \KX},\frac{\partial}{\partial \KY})$,
\begin{align*}
  \dot{\KX}=&\,\dot{\XX}=v_1,\\
  \dot{\KY}=&\,\dot{\YY}-\dot{\XX}=v_2-v_1\\
           = &\,-\KY b
        -\KY^2\frac{b}{\yc}\CCO{ 1-\frac{\KX}{\KX+\yc} }
  +\KY b
         +\KX^2\CCO{1-\frac{a\yc k_1}{\zc^3}+\frac{ak_1}{\zc^2}}
         +\KX\KY \frac{ak_1}{\zc^2}\\
      &\,-\KX^3\frac{a k_1}{\zc^2(\KX+\zc)}\CCO{-\frac{\yc}{\zc}+1}
       -\KX^2\KY  \frac{ak_1^2}{\zc^2(\KX+\zc)}\\
  =&\,\KX^2\CCO{1-\frac{a\yc k_1}{\zc^3}+\frac{ak_1}{\zc^2}}
      +\KX\KY \frac{ak_1}{\zc^2}
      -\KY^2\frac{b}{\yc}
  +\KX^3\frac{a k_1}{\zc^2(\KX+\zc)}\CCO{\frac{\yc}{\zc}-1}\\
       &\,-\KX^2\KY  \frac{ak_1^2}{\zc^2(\KX+\zc)}
      +\KX \KY^2\frac{b}{\yc(\KX+\yc)}.
\end{align*}
We can thus write
	\begin{equation}\label{NILPOT}	
\begin{aligned}
  \dot{\KX}=&\,-\KY b +A(\KX,\KY),\\
  \dot{\KY}=&\,B(\KX,\KY),
\end{aligned}
\end{equation}
where $A$ and $B$ are analytic and their jacobian matrix at $(0,0)$ is
$0$.
In the neighborhood of $(0,0)$, the equation
$0=-\KY b +A(\KX,\KY)$ has the unique solution
$\KY=\Kf(\KX)$, where
%\begin{equation}\label{eq:\Kf(\KX)}
\begin{align*}
\Kf(\KX)=&\frac{
   -\KX^2\CCO{1-\frac{a\yc k_1}{\zc^3}+\frac{ak_1}{\zc^2}}
    -\KX^3\frac{a k_1}{\zc^2(\KX+\zc)}\CCO{\frac{\yc}{\zc}-k_1}
         }{
    b+\KX \frac{ak_1}{\zc^2} -\KX^2  \frac{ak_1^2}{\zc^2(\KX+\zc)}
          }\\
           =&-\frac{1}{b}(1-\frac{a\yc k_1}{\zc^3}+\frac{ak_1}{\zc^2}){u}^{2}+{\frac {a k1 \left( - k1\yc a
+{\zc}^{3}+a k1\zc+b{\zc}^{2}- \yc \zc b
              \right) }{{b}^{2}{\zc}^{5}}}{u}^{3}\\
         & +O \left( {u}^{4} \right).
\end{align*}
%\end{equation}
Let $\KFF(\KX)=B(\KX,\Kf(\KX))$.
Since $A(\KX,\Kf(\KX))=b\Kf(\KX)$ and $B(\KX,\KY)$ has the form
$$B(\KX,\KY)=\KY b-A(\KX,\KY)-\KY b
             -\KY^2\frac{b}{\yc}\CCO{ 1-\frac{\KX}{\KX+\yc} },$$
we have
\begin{align*}
  \KFF(\KX)=&-b\Kf(\KX)\,
             -\Kf^2(\KX)\frac{b}{\yc}\CCO{ 1-\frac{\KX}{\KX+\yc} }\\
                =&{\frac { \left( -k_1\,\yc\,a\zc\,b-{k_1}^{2}\yc
\,{a}^{2}+{\zc}^{3}ak_1+ak_1\,{\zc}^{2}b+{a}^{2}{{
k_1}}^{2}\zc \right) {u}^{3}}{{\zc}^{5}{b}^{2}}}\\
                +&{\frac { \left(
b{\zc}^{2}k_1\,\yc\,a-{\zc}^{5}b-b{\zc}^{3}a{k_1
} \right) {u}^{2}}{{\zc}^{5}{b}^{2}}}
+o(\KX^3).
\end{align*}
Let also $\KGG(\KX)=(\partial A/\partial \KX+\partial B/\partial
\KY)(\KX,\Kf(\KX))$. We have
\begin{align*}
  \partial A/\partial \KX
      =&\,-\KY\frac{ak_1}{\zc^2}
         -2\KX\CCO{1-\frac{a\yc k_1}{\zc^3}+\frac{ak_1}{\zc^2}}
         +2\KX\KY\frac{ak_1}{\zc^2(\KX+\zc)}\\
        &\, +3\KX^2\frac{a k_1}{\zc^2(\KX+\zc)}\CCO{-\frac{\yc}{\zc}+1}
        -\KX^3\frac{a k_1}{\zc^2(\KX+\zc)^2}\CCO{-\frac{\yc}{\zc}+1}
           -\KX^2\KY\frac{ak_1}{\zc^2(\KX+\zc)^2},\\
%\frac{1}{(\KX+\zc)^2}\frac{a\yc k_1}{\zc^3},\\
\partial B/\partial \KY
      =&\,-2\KY\frac{b}{\yc}\CCO{ 1-\frac{\KX}{\KX+\yc} }
      +\KX\frac{ak_1}{\zc^2}
      -\KX^2  \frac{ak_1}{\zc^2}.
\end{align*}
Replacing $\KY$ by $\Kf(\KX)$ yields
\begin{align*}
  \KGG(\KX)=&\,\KX\left\lgroup
       -2\CCO{1-\frac{a\yc k_1}{\zc^3}+\frac{ak_1}{\zc^2}}
               +\frac{ak_1}{\zc^2}
                    \right\rgroup\\%+o(\KX).
   &\,+\KX^2\left\lgroup
\frac{1}{b}\CCO{1-\frac{a\yc k_1}{\zc^3}+\frac{ak_1}{\zc^2}}\frac{ak_1}{\zc^2}
+3\frac{a k_1}{\zc^3}\CCO{-\frac{\yc}{\zc}+k_1}
      \right\rgroup+o(\KX^2).
\end{align*}
% \begin{align*}
%   (\partial A/\partial \KX)(\KX,\Kf(\KX))
%       =&\,-2\KX\CCO{1-\frac{a\yc k_1}{\zc^3}+\frac{ak_1}{\zc^2}}+o(\KX),\\
% (\partial B/\partial \KY)(\KX,\Kf(\KX))
%       =&\,\KX\frac{ak_1}{\zc^2}+o(\KX).
% \end{align*}
\textbf{Case 1:} If $1-\frac{a\yc k_1}{\zc^3}+\frac{ak_1}{\zc^2}\neq 0$, then
\begin{align*}
  \KFF(\KX)=&u^2(1-\frac{a\yc k_1}{\zc^3}+\frac{ak_1}{\zc^2})+o(\KX^2),\\
\intertext{and}
  \KGG(\KX)=&\,\KX\left\lgroup
       -2\CCO{1-\frac{a\yc k_1}{\zc^3}+\frac{ak_1}{\zc^2}}
               +\frac{ak_1}{\zc^2}
                    \right\rgroup\\%+o(\KX).
   &\,+\KX^2\left\lgroup
\frac{1}{b}\CCO{1-\frac{a\yc k_1}{\zc^3}+\frac{ak_1}{\zc^2}}\frac{ak_1}{\zc^2}
+3\frac{a k_1}{\zc^3}\CCO{-\frac{\yc}{\zc}+k_1}
      \right\rgroup+o(\KX^2).
\end{align*}
We can now  apply \cite[Theorem 3.5]{dumortier-llibre-artes} to system \eqref{NILPOT}.
Since the coefficient of $\KX^2$
in $\KFF(\KX)$ is nonzero,
we deduce from Part (4)-(i1) of \cite[Theorem 3.5]{dumortier-llibre-artes}:

%%%% LEM : NILPOTENT = CUSP
\begin{lemma}%\label{lem:cusp}
If $\Ec$ is a
nilpotent equilibrium of \eqref{eq:simple} in the positive quadrant
 $]0,+\infty[\times ]0,+\infty[$, and if $1-\frac{a\yc k_1}{\zc^3}+\frac{ak_1}{\zc^2}\neq 0$, then $\Ec$ is a cusp, that is,
its phase portrait consists of two hyperbolic sectors and two
separatrices. In this case, the index of $\Ec$ is 0.
\end{lemma}
%%%%
\textbf{Case 2:} if $1-\frac{a\yc k_1}{\zc^3}+\frac{ak_1}{\zc^2} = 0$, then
\begin{align*}
\Kf(\KX)=&{\frac {a k1 \left( - k1\yc a
+{\zc}^{3}+a k1\zc+b{\zc}^{2}- \yc \zc b
 \right) }{{b}^{2}{\zc}^{5}}}{u}^{3}+O \left( {u}^{4} \right)   \\
 =&-{\frac {1}{b\zc}}u^{3}+O \left( {u}^{3} \right),\\
  \KFF(\KX)=&\frac{1}{\zc} u^3+o(\KX^3),\\
\intertext{and}
  \KGG(\KX)=&\KX\left\lgroup
       \frac{ak_1}{\zc^2}
                    \right\rgroup
   \,+\KX^2\left\lgroup
3\frac{a k_1}{\zc^3}\CCO{-\frac{\yc}{\zc}+k_1}
      \right\rgroup+o(\KX^2).
\end{align*}
Again, we apply  \cite[Theorem 3.5]{dumortier-llibre-artes} to System \eqref{NILPOT}.
Since the coefficient of $\KX^3$
in $\KFF(\KX)$ is positive,
we deduce from Part (4)-(ii) of \cite[Theorem 3.5]{dumortier-llibre-artes}:

%%%% LEM : NILPOTENT = SADDLE
\begin{lemma}%\label{lem:cusp}
If $\Ec$ is a
nilpotent equilibrium of \eqref{eq:simple} in the positive quadrant
 $]0,+\infty[\times ]0,+\infty[$, and if $1-\frac{a\yc k_1}{\zc^3}+\frac{ak_1}{\zc^2} = 0$, then $\Ec$ is a saddle point. In this case, the index of $\Ec$ is -1.
\end{lemma}

%\bigskip
           % CAS CENTRE POUR LE SYSTÈME LINÉARISÉ
           \subsubsection*{$\mbox{c}_2$. The case of a center of the linearized
          vector field}
%$\bullet$ {\em Hopf bifurcation and center of the linearized
   %      vector field}
The point $\Ec$ is a {center} of the linear part of $\vectorfield$ if
the Jacobian $\jcb{\xc,\yc}$ has
        purely imaginary eigenvalues      	
	$\pm i\sqrt{p}$, that is, when $p>0$ and $s=0$.
	Again, this case is nonempty.
Let us denote
\begin{equation}
  \label{eq:b0}
  b_0=1-2\xc+\frac{a\yc k_1}{\zc^2}.
\end{equation}
With the notations of \eqref{eq:zc-cc}, we have $p>0$ and $s=0$ if,
and only if,
\begin{equation}
           \label{eq:linearlycenter}
           b=b_0<a\cc.
         \end{equation}
Note that $\xc$, $\yc$, as well as
$b_0$, $a$, $\cc$, and the sign of $p$
do not depend on the parameter $b$, and that $s=b-b_0$.
Let us fix all parameters except $b$, and assume that
$\discrimC<0$, that is,
the eigenvalues of $\jcb{\xc,\yc}$ %the roots of $\charac$
are
$$\frac{-s\pm i\sqrt{4p-s^2}}{2}.$$
These eigenvalues
cross the imaginary axis at
speed $-1/2$ when $b$ passes through the value
$b_0$.
Let us denote $\Kc=a\cc$.
By \eqref{eq:v1} and \eqref{eq:v2}, we have
\begin{align*}
  v_1=&\,\XX b_0 %\CCO{1-2\xc-\frac{a\yc k_1}{\zc}}
        -\YY \Kc
     -\XX^2\CCO{1-\frac{a\yc k_1}{\zc^3}}
        -\XX\YY\frac{ak_1}{\zc^2}
     -\XX^3\frac{a\yc k_1}{\zc^3(\XX+\zc)}
        +\XX^2\YY\frac{ak_1}{\zc^2},\\
  v_2=&\,(\XX -\YY) b-\frac{b}{\yc}(\XX-\YY)^2\CCO{ 1-\frac{\XX}{\XX+\yc} }.
\end{align*}
Let us denote by $(\ii,\jj)$ the standard basis of $\R^2$.
In this basis,
the matrix of the linear part $\varphi$ of $(\XX,\YY)\mapsto(v_1,v_2)$ is
$$
A(b)=\begin{pmatrix}
b_0&-\Kc\\
b&-b
\end{pmatrix}.
$$
Let
%\begin{equation*}
\begin{alignat*}{3}
\dzo=&\sqrt{\det A(b_0)}=\sqrt{b_0(\Kc-b_0)},
      &\quad   &\DZO=\frac{\Kc-b_0}{\dzo}=\sqrt{\frac{\Kc-b_0}{b_0}},\\
  \uu=&\ii+\jj,
      &\quad          &\vv=\frac{1}{\dzo}\varphi(\uu)
 =-\DZO\ii.
\end{alignat*}
%\end{equation*}
The matrix of $\varphi$ in the basis $(\uu,\vv)$ is
$$
\tilde{A}(b)=\begin{pmatrix}
0&-\frac{b}{b_0}\dzo\\
\dzo&b_0-b
\end{pmatrix}.
$$
%We make the change of variables
The coordinates $(\KX,\KY)$ in the basis $(\uu,\vv)$ satisfy
$\KX=\YY,\quad \KY=\frac{1}{\DZO}(\YY-\XX)$,
$\XX=\KX-\KY\DZO$, $\YY=\KX$.
The coordinates of $\vectorfield$ in the basis
$(\frac{\partial}{\partial \KX}, \frac{\partial}{\partial \KY})$ are
\begin{align*}
  \dot{\KX}&\,=v_2=-b\DZO\KY
       +\frac{b}{\yc}\DZO^2\KY^2
          \CCO{ 1-\frac{\KX-\KY\DZO}{\KX-\KY\DZO+\yc} },\\
  \dot{\KY}=&\,\frac{1}{\DZO}(v_2-v_1)\\
      = &\, -b\KY
       +\frac{b}{\yc}\DZO\KY^2
          \CCO{ 1-\frac{\KX-\KY\DZO}{\KX-\KY\DZO+\yc} }\\
   %%%%%%%%%%%
   & \, -\frac{1}{\DZO}\Biggl\lgroup
         (\KX-\KY\DZO) b_0 -\KX \Kc
     -(\KX-\KY\DZO)^2\CCO{1-\frac{a\yc k_1}{\zc^3}}
        -\KX(\KX-\KY\DZO)\frac{ak_1}{\zc^2}\\
    &\, \phantom{-\frac{1}{\DZO}\Biggl\lgroup}
        -(\KX-\KY\DZO)^3\frac{a\yc k_1}{\zc^3(\KX-\KY\DZO+\zc)}
        +\KX(\KX-\KY\DZO)^2\frac{ak_1}{\zc^2}\Biggr\rgroup.\\
\intertext{In particular, for $b=b_0$,}
  \dot{\KX}=&\,-\dzo\KY
       +\frac{\Kc-b_0}{\yc}\KY^2
          \CCO{ 1-\frac{\KX-\KY\DZO}{\KX-\KY\DZO+\yc} },\\
 \dot{\KY} =&\, \dzo\KX
          +\frac{\Kc-b_0}{\yc}\KY^2
          \CCO{ 1-\frac{\KX-\KY\DZO}{\KX-\KY\DZO+\yc} }\\
    &\, +\frac{1}{\DZO}\Biggl\lgroup
      +(\KX-\KY\DZO)^2\CCO{1-\frac{a\yc k_1}{\zc^3}}
        +\KX(\KX-\KY\DZO)\frac{ak_1}{\zc^2}\\
    &\, \phantom{-\frac{1}{\DZO}\Biggl\lgroup}
        +(\KX-\KY\DZO)^3\frac{a\yc k_1}{\zc^3(\KX-\KY\DZO+\zc)}
        -\KX(\KX-\KY\DZO)^2\frac{ak_1}{\zc^2}\Biggr\rgroup.
\end{align*}

	%%%%%%%%%%%%%%%%%%%%%%%%%%%%%%%%%
	\subsection{Existence of a globally asymptotically stable equilibrium point}
	When $m=0$, in the case \eqref{item:m0-nosol} of
              \Cref{prop:m=0nbdepoints},
	we have seen that \eqref{eq:simple} has
	no cycle, %in $invrg$,
	because the compact set delimited by
	a cycle would contain a critical point,
	see \cite[Theorem V.3.8]{bhatia-szego}.
	As the compact set $\invrg$ is invariant and contains
	all equilibrium points of the
        open quadrant $]0,+\infty[\times ]0,+\infty[$,
	all trajectories starting in the quadrant $\R_+\times\R_+$
	converge to $E_1$ or $E_2$ ($E_0$ is excluded because it is an unstable
	node).
	On the $x$ axis,
        we have $\dot{y}=0$ and $x$ satisfies the logistic equation
	$\dot{x}=x(1-x)$, thus, for $x(0)>0$, $x(t)$ converges to 1, i.e.,
	$(x(t),y(t))$ converges to $E_1$.
	On the other hand,
	for $0<y<k_2+x$, we have $\dot{y}>0$, thus, if $y(0)>0$,
        $(x(t),y(t))$ cannot
	converge to $E_1$, it converges necessarily to $E_2$.

	%%%%%%%%%%%% THEO FONCTION DE LYAPUNOV %%%%%%%
	\begin{theorem}\label{theo:globstable}
		A sufficient condition for the existence of
		a globally asymptotically stable
		equilibrium point $\Ec=(\xc,\yc)$ in the open quadrant
		$]0,+\infty[\times ]0,+\infty[$
                (equivalently, in the interior of $\invrg$)
		is that
		\begin{equation}
			\label{eq:2m+k}
			\big\lgroup 2m+k_1\geq 1\big\rgroup \text{ and }
			\big\lgroup (m>0)\text{ or }
			\bigl(4ak_2\leq (1-k_1-a)^2+4k_1\bigr)\big\rgroup.
		\end{equation}
	\end{theorem}
	%%%%%%%%%%
	\begin{proof}
	Let
	$\Ec=(\xc,\yc)\in\invrg$ be an equilibrium point in the interior of
	$\invrg$.
	Let us denote
	\begin{equation*}
		\roh(x)=\frac{a(x-m)}{k_1+x-m},
	\end{equation*}
	and let us set
	\begin{equation*}
		V(x,y)=\int_\xc^x \frac{u-\xc}{(k_2+u-m)\roh(u)}\dd u
		+\frac{1}{b}\int_\yc^y \frac{v-\yc}{v}\dd v.
	\end{equation*}
	Then, using \eqref{eq:xstar2} and \eqref{eq:ystar2}, we have
	\begin{align*}
		\dot{V}=&\,\frac{x-\xc}{(k_2+x-m)\roh(x)}\,\dot{x}
		+\frac{1}{b}\frac{y-\yc}{y}\,\dot{y}\\
		=&\,\frac{x-\xc}{k_2+x-m}
		\CCO{
			\frac{x(1-x)}{\roh(x)}-\frac{a(x-m)}{k_1+x-m}\frac{1}{\roh(x)}y
		}
		+\frac{1}{b}(y-\yc)b\CCO{1-\frac{y}{k_2+x-m}}\\
		=&\,\frac{x-\xc}{a(k_2+x-m)}\CCO{
			\frac{x(1-x)(k_1+x-m)}{x-m}-\yc
		}
		-\frac{(x-\xc)(y-\yc)}{k_2+x-m}\\
		&\,+(y-\yc)\CCO{
			\frac{\yc}{k_2+\xc-m}-\frac{y}{k_2+x-m}
		}\\
		=&\,\frac{x-\xc}{a(k_2+x-m)}\CCO{
			\frac{x(1-x)(k_1+x-m)}{x-m}-\frac{\xc(1-\xc)(k_1+\xc-m)}{\xc-m}
		}\\
		&\,-\frac{(x-\xc)(y-\yc)}{k_2+x-m}\\
		&\,+(y-\yc)\,{
			\frac{\yc(k_2+x-m)-y(k_2+\xc-m)}{(k_2+\xc-m)(k_2+x-m)}
		}.\\
\intertext{Let us denote $\gfun(x)=x(1-x)(k_1+x-m)/(x-m)$. Then }
		\dot{V}=&\,\frac{x-\xc}{a(k_2+x-m)}\CCO{\gfun(x)-\gfun(\xc)}
		-\frac{(x-\xc)(y-\yc)}{k_2+x-m}\\
		&\,+(y-\yc)\,
		\frac{(\yc-y)(k_2-m)+\yc x-y\xc}{(k_2+\xc-m)(k_2+x-m)}\\
		=&\,\frac{x-\xc}{a(k_2+x-m)}\CCO{\gfun(x)-\gfun(\xc)}
		-\frac{(x-\xc)(y-\yc)}{k_2+x-m}\\
		&\,+\frac{y-\yc}{\yc}\,\frac{(\yc-y)(\xc+k_2-m)+\yc(x-\xc)}{k_2+x-m}\\
		=&\,\frac{x-\xc}{a(k_2+x-m)}\CCO{\gfun(x)-\gfun(\xc)}
		+\frac{y-\yc}{\yc}\,\frac{(\yc-y)(\xc+k_2-m)}{k_2+x-m}\\
		=&\,\frac{1}{k_2+x-m}\CCO{
			\frac{x-\xc}{a}\CCO{\gfun(x)-\gfun(\xc)} - (y-\yc)^2
		}.
	\end{align*}
	For $x\geq m$, a sufficient condition for $\dot{V}$ to be
	negative when $(x,y)\not=(\xc,\yc)$ is that $\gfun$ be nonincreasing.
	Let us make the change of variable $X=x-m$.
	We have
	\begin{equation*}
		\gfun(x)=\frac{(X+m)(1-X-m)(X+k_1)}{X},
	\end{equation*}
	which leads to
	\begin{equation*}
		\gfun'(x)=\frac{-2X^3+(1-2m-k_1)X^2-k_1(m-m^2)}{X^2}.%=:\frac{h(X)}{X^2}.
	\end{equation*}
	% If $m>0$, we have $h(0)<0$ and
	% \begin{equation*}
	%   h'(X)=-6X\CCO{X-\frac{1-2m-k_1}{3}},
	% \end{equation*}
	Thus, if $2m+k_1\geq 1$, %\eqref{eq:2m+k} is satisfied,
	$\gfun'(X)$ remains negative for $X> 0$, i.e., for $x> m$.
	Thus, for $x> m$, under the assumption \eqref{eq:2m+k},
	$\dot{V}$ is negative.
	
	% When $m=0$, the function $g$ becomes
	% $g(x)=(1-x)(x+k_1)$, thus $g'(x)=-2x+1-k_1$ is negative for all $x>0$
	% if $k_1\geq 1$, which is again Condition \eqref{eq:2m+k}.

	We have seen that the first part of \eqref{eq:2m+k} implies
	that the equilibrium point $\Ec$, if it exists,
	is globally asymptotically stable.
	Note that Condition \eqref{eq:2m+k} is independent of the coordinates of $\Ec$,
	and the global stability implies that the equilibrium point $\Ec$, if
	it exists, is unique.
	
	The second part of \eqref{eq:2m+k} is a necessary and sufficient
	condition for the existence of such an equilibrium point.
	
	When $m>0$, we already know that there exists at least one
	equilibrium point in $\invrg$.
	Actually,
	Condition \eqref{eq:2m+k} implies that the coefficient $\alpha_2=a+k_1-1+2m$ of
	\eqref{eq:Rcoeffs} is positive.
	Thus, when $m>0$, \eqref{eq:2m+k} is a particular
	case of \eqref{cond:n=1} in \Cref{theo:nbdepoints}.

	When $m=0$, by \Cref{prop:m=0nbdepoints}-\eqref{item:m0-nosol},
	since $\alpha_2>0$,
	there exists an equilibrium point in the interior of $\invrg$
	if, and only if,
	\eqref{eq:discriminant} is satisfied.
	\end{proof}
	
	%%%%%%%%%%%%%%%%%%%%%%%%%%%%%%%%%
	\subsection{Cycles}	
	Let us investigate the existence of periodic orbits of
	\eqref{eq:simple}.
	By \Cref{theo:Ainvariant} such orbits can take place only
	in $\invrg$.
	\subsubsection{Refuge free case ($m=0$)}
	This case has been studied by M.A.~Aziz-Alaoui and M.~Daher-Okiye \cite{Daher-Aziz03bologn}, but we add some new results.
	%%%%%%%%%
	\begin{lemma}\label{lem:0equilib0cycle}
		In the cases \eqref{item:m0-nosol}
		and \eqref{item:m0-2sol}
		of \Cref{prop:m=0nbdepoints},
		that is, when \eqref{eq:simple} has 0 or 2 equilibrium points in the open quadrant $]0,+\infty[\times]0,+\infty[$,
		the system \eqref{eq:simple} has
		no limit cycle.
		On the other hand, in the case  \eqref{item:m0-1sol} of  Theorem  \cref{prop:m=0nbdepoints}, that is, when \eqref{eq:simple} has 1 equilibrium point in the open quadrant $]0,+\infty[\times]0,+\infty[$, if furthermore $s<0$ and $p>0$, the system \eqref{eq:simple} has at least one limit cycle.
	\end{lemma}
	
	%%%%%%%%%
	\begin{proof}
	In the case \eqref{item:m0-nosol},
	the only equilibrium points of \eqref{eq:simple} in $\R_+\times\R_+$
	are the trivial points $E_0$, $E_1$, and $E_2$, on the axes.
	Thus \eqref{eq:simple} has
	no cycle, %in $invrg$,
	because the compact set delimited by
	a cycle would contain a critical point,
	see \cite[Theorem V.3.8]{bhatia-szego}.
	
	In the case \eqref{item:m0-2sol},
	if there was a cycle inside $\invrg$, we could apply
	the Poincar\'e-Hopf Index Theorem to the compact manifold whose
	boundary is delineated by this cycle
	(see \cite{ma-wang} for a version of this theorem when the vector field
	is tangent to the boundary).
	Denoting $N$ the number of nodes or
	foci and $S$ the number of saddles in the
	open quadrant $]0,+\infty[\times ]0,+\infty[ $, we would have
	$N-S=1$.
	But \Cref{theo:poincare-hopf} shows that $N-S=0$, a contradiction.
	
	In the case \eqref{item:m0-1sol}, if $s<0$ and $p>0$, the
        system \eqref{eq:simple} has an unstable equilibrium point.
        From \Cref{theo:Ainvariant} and Poincar\'e-Bendixson Theorem, there exists at least one limit cycle around this equilibrium.
	\end{proof}
	Note that the conditions of \Cref{lem:0equilib0cycle} do not involve the value of $b$.
	Using Bendixson-Dulac criterion, M.A.~Aziz-Alaoui and M.~Daher-Okiye obtain another criterion:
	\begin{lemma}\cite[Theorem 7]{Daher-Aziz03bologn}
	if $b+k_1\geq1$, then the system \eqref{eq:simple} has
		no limit cycle.
	\end{lemma}
	
	%If $\discrimQ=0$, ...................

	\subsubsection{Case with refuge ($m>0$)}
	By \Cref{theo:globstable}, if Condition \eqref{eq:2m+k} is
	satisfied, there can be no periodic orbits.
	
	Let us now give some sufficient conditions for the absence of
	periodic orbits,
	using Bendixson-Dulac criterion.
	Let us denote by $\cf(x,y)$ and $\cg(x,y)$ the coordinates of the
	vector field in \eqref{eq:simple}.
	For a Dulac function, we choose
	\begin{equation*}
		\Dulac(x,y)=x+k_1-m.
	\end{equation*}
	Let us look for conditions that ensure that
	$\frac{\partial(\cf \Dulac)}{\partial x}+\frac{\partial(\cg \Dulac)}{\partial
		y}<0$ in $\invrg$.
	We have
	\begin{align*}
		%\label{eq:dulac-bendixson}
		\frac{\partial(\cf \Dulac)}{\partial x}(x,y)%+\frac{\cg \Dulac}{\partial y}(x,y)\\
		&=-3x^2+2(1-k_1+m)x+k_1-m-ay,\\
		\frac{\partial(\cg \Dulac)}{\partial y}(x,y)
		&=\frac{b(x+k_1-m)(x+k_2-m-2y)}{x+k_2-m}.
	\end{align*}
	For $(x,y)\in\invrg$, we have
	\begin{equation*}
		\frac{\partial(\cf \Dulac)}{\partial x}(x,y)
		<-3m^2+2(1-k_1+m)x+k_1-m-ak_2.
	\end{equation*}
	Since the maximum of $-3m^2+m$ is $1/12$ and the maximum of $-m^2+m$
	is $1/4$, we deduce:
	\begin{align*}
		1-k_1+m>0\Rightarrow
		\frac{\partial(\cf \Dulac)}{\partial x}(x,y)
		&< -3m^2+m-k_1-ak_2+2\\
		&\leq 2+\frac{1}{12}-k_1-ak_2,\\
		1-k_1+m<0\Rightarrow
		\frac{\partial(\cf \Dulac)}{\partial x}(x,y)
		&< -m^2+m(1-2k_1)-k_1-ak_2\\
		&< -m^2-m -k_1-ak_2<0.
	\end{align*}
	In particular, a condition that ensures that
	$\frac{\partial(\cf \Dulac)}{\partial x}<0$ in $\invrg$ is
	\begin{equation}
		\label{eq:DBf}
		%\big\lgroup
		(k_1>1+m)
		\text{ or }(ak_2+k_1>2+\frac{1}{12}). %\big\rgroup.
	\end{equation}
	On the other hand,
	for $(x,y)\in\invrg$,
	$\frac{\partial(\cg \Dulac)}{\partial y}(x,y)$
	has the same sign as $x+k_2-m-2y$,
	and we have $x+k_2-m-2y<1-m-k_2$.
	Thus a sufficient condition for  $\frac{\partial(\cg \Dulac)}{\partial
		x}<0$ in $\invrg$ is
	\begin{equation}\label{eq:DBg}
		k_2>1-m.
	\end{equation}
	The same technique does not provide any sufficient condition for
	$\frac{\partial(\cf \Dulac)}{\partial x}+\frac{\partial(\cg \Dulac)}{\partial
		y}>0$ in $\invrg$. So, our next result concerning the absence of cycles is:
	%%%%%%%%%%%%
	\begin{lemma}\label{lem:Dulac0}
		A sufficient condition for \eqref{eq:simple} to have no
		periodic solution is
		$$
		\big\lgroup k_2>1-m\big\rgroup \text{ and }
		\big\lgroup (k_1>1+m)
		\text{ or }(ak_2+k_1>2+\frac{1}{12})\big\rgroup.
		$$
	\end{lemma}
%%%%%%%%%
	
Now, we consider the existence of limit cycles which are not occuring
from a Hopf bifurcation.  The special configuration of the existence
of a limit cycle enclosing three equilibrium points is numerically
investigated. In particular, when the system parameters satisfy
$a=0.5,k1=0.08,k2=0.2,b=0.1,m=0.0025,$ then three hyperbolic
equilibrium points exist, namely,
$E_1^*=(0.0222589;0.2197589)$, $E_2^*=(0.0299525;0.2274525)$,
$E_3^*=(0.3702886;0.5677886)$.
They define respectively a stable focus, a saddle point and an unstable
focus. Accordingly to the Poincar\'e index theorem,  the sum of the
corresponding indexes is equal to $1$.

The numerical simulations show that there exists a limit cycle, which is hyperbolic and stable, see \Cref{fig:3points-cycle}.

	%%%%%%%%%%%%%%%%%%%%%%%%%%%%%%%%%%%%%%%%%%%%%%%%%%%%%
	\section{Stochastic model}\label{sec:stoch}
	We now study the dynamics of the system
        \eqref{eq:stochastic}, with
	initial conditions $x_0>0$ and $y_0>0$.
        In the case when $m=0$ and $k_1=k_2$,
        the persistence and boundedness of
        solutions have been investigated in
        by Ji, Jiang and Shi in \cite{ji-jiang-shi}.
        A similar model has been studied by Fu, Jiang, Shi, Hayat and
        Alsaedi
        in \cite{fu-jiang-shi-hayat-alsaedi}.

\subsection{Existence and uniqueness of the positive global solution}
\label{subsec:existence}
	%%%%%%%%%%%%%%%%%%
\begin{theorem}\label{theo:stoch-existence}
For any initial condition $(x_0,y_0)\in\mathbb{R}_{+}^{2} $, the system
 \eqref{eq:stochastic} admits a unique solution $(x(t),y(t))$,
 defined for all $t\geq 0$ a.s.
        and this solution remains in $]0,+\infty[\times ]0,+\infty[$.
Furthermore, if $(x_0,y_0)\in ]0,+\infty[\times]0,+\infty[$,
this solution remains in $ ]0,+\infty[\times]0,+\infty[$,
whereas, if $(x_0,y_0)$ belongs to one of the axis $\R_{+}\times\{0\}$
or  $\{0\}\times\R_{+}$, it remains on this axis.
\end{theorem}
\begin{proof} %of{theorem \cref{theo:stoch-existence}}
Since the coefficients of  \eqref{eq:stochastic} are locally
Lipschitz,
uniqueness of the solution until explosion time is guaranteed for any
initial condition.

Let us now prove global existence of the solution.

The case when $(x_0,y_0)\in\Bigl(\R_{+}\times\{0\}\Bigr)\cup
\Bigl(\{0\}\times\R_{+}\Bigr)$ is trivial because both equations in
\eqref{eq:stochastic} become independent,
for example if $y_0=0$ with $x_0\not=0$, we have $y(t)=0$ for all
$t\geq 0$, and $x$ is a solution to the stochastic
logistic equation
\begin{equation*}
  dx(t)=x(t)(1-x(t))dt + \sigma_{1} x(t) dw_{1}(t)
\end{equation*}
which is well known (see \Cref{subsec:comparison}), thus $x(t)$
is defined for every $t\geq 0$.

Assume now that $x_0>0$ and $y_0>0$.
Since the coordinate axes are stable by \eqref{eq:stochastic}, we deduce,
applying locally the comparaison theorem for SDEs
(see \cite[Theorem 1]{Ferreyra-sundar}, this theorem is given
for globally Lipschitz coefficients),
that the solution to
	\eqref{eq:stochastic} remains in $]0,+\infty[\times ]0,+\infty[$
	until its explosion time.
	
	Let $\tau_e$ be the explosion time of the solution to
	\eqref{eq:stochastic}.
       To show that $\tau_e =\infty$, we adapt the proof of
       \cite{Dalal}.
       Let $k_0 >0$ be large enough, such that $\left( x_{0},y_{0}\right) \in [\frac{1}{k_{0}},k_{0}]\times[\frac{1}{k_{0}},k_{0}]$. For each integer $k \geq k_0 $ we define the stopping time $$\tau_{k}=\inf \Big\{t\in [0,\tau_{e}):x\notin (\frac{1}{k},k) \text{\ or \ } y\notin (\frac{1}{k},k)\Big\}.$$
	The sequence $(\tau_k)$ is increasing as $k \rightarrow
        \infty$. Set $\tau_\infty = \lim_{k\rightarrow \infty }\tau_k
        $, whence $\tau_\infty\leq \tau_e$, (in fact, as
        $(x(t),y(t))>0$ a.s., we have $\tau_{\infty}=\tau_e$). It
        suffices to prove that $\tau_\infty =\infty$ a.s.. Assume that
        this statement is false, then there exist $T>0$ and
        $\varepsilon \in \left] 0,1\right[ $ such that $\prob \left(
          \{\tau_{\infty}\leq T\}\right) >\varepsilon$. Since
        $(\tau_k)$ is increasing we have $$\prob \left( \{\tau_{k}\leq
          T\}\right) >\varepsilon.$$
Consider now the positive definite function $V$ : $ ]0,+\infty[\times ]0,+\infty[\rightarrow ]0,+\infty[\times ]0,+\infty[$ given by $$V(x,y)=(x+1-\log x)+(y+1-\log y).$$
	Applying It\^o's formula, we get
	\begin{align*}
		dV(x,y)=&\Big[(x-1)(1-x-\frac{ay(x-m)}{k_{1}+x-m})+\frac{\sigma_{1}^{2}}{2}+b(y-1)(1-\frac{y}{k_{2}+x-m})+\frac{\sigma_{2}^{2}}{2}\Big]dt\\
		+&\sigma_{1}(x-1)dW_{1}+ \sigma_{2}(y-1)dW_{2}.
	\end{align*}
	The positivity of $x(t)$ and $y(t)$ implies
	\begin{align*}
		dV(x,y)&\leq\Big(2x+ay+\frac{\sigma_{1}^{2}+\sigma_{2}^{2}}{2}+by+\frac{y}{k_{2}}\Big)dt+
		\sigma_{1}(x-1)dW_{1}+ \sigma_{2}(y-1)dW_{2}\\
		&\leq\Big(2x+(a+b+\frac{1}{k_2})y+\frac{\sigma_{1}^{2}+\sigma_{2}^{2}}{2}\Big)dt+
		\sigma_{1}(x-1)dW_{1}+ \sigma_{2}(y-1)dW_{2}.
	\end{align*}
	Denote $c_1=a+b+\frac{1}{k_2}$, $c_2=\frac{\sigma_{1}^{2}+\sigma_{2}^{2}}{2}$.
	Using \cite[lemma 4.1]{Dalal}, we can write
	\begin{align*}
		2x+c_1y\leq& 4(x+1-\log x)+2c_1(y+1-\log y)\\
		\leq& c_3 V(x,y), %\text { where } c_3=\max(4,2c_1).
	\end{align*}
        where $c_3=\max(4,2c_1)$.
	Hence, denoting $c_4=\max(c_2,c_3)$,
	\begin{align*}
		dV(x,y)\leq&(c_2+c_3V(x,y))dt+\sigma_{1}(x-1)dW_{1}+\sigma_{2}(y-1)dW_{2}\\
		\leq& c_4(1+V(x,y))
                      dt+\sigma_{1}(x-1)dW_{1}+\sigma_{2}(y-1)dW_{2}. %, \text { where } c_4=\max(c_2,c_3).
	\end{align*}
	Integrating both sides from $0$ to $\tau_k \wedge T$, and taking expectations, we get
	\begin{equation*}
		\expect V(x(\tau_{k}\wedge T),y(\tau_{k}\wedge T))
		\leq V(x_{0},y_{0})+c_{4}T+c_{4}\int_{0}^{T} \expect V(x(\tau_{k}\wedge t),y(\tau_{k}\wedge t)dt.
	\end{equation*}
	By Gronwall's inequality, this yields
	\begin{equation}\label{*}
		\expect V(x(\tau_{k}\wedge T),y(\tau_{k}\wedge T))\leq c_5,
	\end{equation}
	where $c_5$ is the finite constant given by
	\begin{equation}\label{cond}
		c_5=(V(x_{0},y_{0})+c_{4}T)e^{c_{4}T}.
	\end{equation}
	Let $\Omega_k=\{\tau_k\leq T\}$. We have $\prob (\Omega_k)\geq \varepsilon$, and for all $\omega \in \Omega_k$, there exists at least one element of ${x(\tau_{k},\omega),y(\tau_{k},\omega)}$ which is equal either to $k$ or to $\frac{1}{k}$, hence $$V(x(\tau_{k}),y(\tau_{k}))\geq (k+1-\log k)\wedge (\frac{1}{k}+1+\log k).$$
	Therefore, by \eqref{*},
	$$c_{5}\geq \expect[1_{\Omega_{k}}(\omega)V(x(\tau_{k},\omega),y(\tau_{k},\omega)]\geq \varepsilon \Big[(k+1-\log k)\wedge (\frac{1}{k}+1+\log k)\Big],$$
	where $1_{\Omega_{k}}$ is the indicator function of $\Omega_{k},$. Letting $k\rightarrow \infty$, we get
	$c_{5}=\infty$, which contradicts \eqref{cond}, So we must have $\tau_{\infty}= \infty$ a.s.
	\end{proof}
	
%%%%%%%%%%%
\begin{remark}\label{rem:altproof}
  An alternative proof of non explosion in finite time
  can be obtained by using the comparison theorem,
since $0\leq x(t)\leq z_1(t)$ and $0\leq y(t)\leq z_2(t)$ a.s.~for every
$t\geq 0$, where $z_1$ and $z_2$ are geometric Brownian motions, with
\begin{equation*}
  dz_1(t)=z_1(t)dt+\sigma_1z_1(t)dW_1(t)
  \text{ and }
  dz_2(t)=b z_2(t)dt+\sigma_2z_2(t)dW_2(t).
\end{equation*}
\end{remark}

%%%%%%%%%%%%% COMPARISON RESULTS
\subsection{Comparison results}\label{subsec:comparison}
 In this section, we compare the dynamics of \eqref{eq:stochastic}
 with some simpler models, in view of applications to the long time
 behaviour of the solutions to \eqref{eq:stochastic}.	

 Applying locally the comparaison theorem for SDEs
	(see \cite[Theorem 1]{Ferreyra-sundar}, this theorem is given
        for globally Lipschitz coefficients),
	we have, for every $t\geq 0$,
        \begin{equation}
          \label{eq:comparison-xu}
          0\leq x(t)\leq \xu(t)
          \text{ a.s.}
        \end{equation}
        where $\xu$ is the solution to the stochastic logistic equation
        (also called stochastic Verhulst equation)
        with initial condition $x_0$:
	\begin{equation}\label{eq:xu}%\label{logistic}
          d\xu(t)=\xu(t)(1-\xu(t))dt+ \sigma_{1} \xu(t) dw_1(t),
          \quad \xu(0)=x_0.
	\end{equation}
	The process $\xu$ is well known and can be written explicitely, see
	\cite[page 125]{kloeden-platen}:
        \begin{equation*}
          \xu(t)=\frac{e^{ \bigl( 1-\frac{\sigma_1^2}{2}\bigr)t+\sigma_1w_1(t)}}{
                \frac{1}{x_0}+\int_0^te^{ \bigl(
                  1-\frac{\sigma_1^2}{2}  \bigr)s+\sigma_1w_1(s)}ds}.
            \end{equation*}
By \cite[Lemma 2.2]{liu-shen2015}, $\xu$ is uniformly
bounded in $L^p$ for every $p>0$. Thus, by \eqref{eq:comparison-xu},
for every $p>0$, there exists a constant $K_p$ such that
\begin{equation}\label{moment}
		\sup_{t \geq 0} \expect \left(x(t)\right)^p < K_p.
 \end{equation}
Using again the comparison theorem, we get, for every $t\geq 0$,
         \begin{equation}\label{eq:comparison-yu}
          0\leq y(t)\leq \yu(t),
         \end{equation}
         where $\yu$ is the solution to
         \begin{equation}\label{eq:yu}
           d\yu(t)=b \yu(t)\CCO{1-\frac{\yu(t)}{k_2+\xu(t)}}dt
           +  \sigma_{2} \yu(t) dw_{2}(t),
           \quad \yu(0)=y_0,
         \end{equation}
         which can be explicited with the help of $\xu$:
         \begin{equation}\label{eq:yu-explicit}
         \yu(t)=\frac{e^{ \bigl( b-\frac{\sigma_2^2}{2}\bigr)t+\sigma_2w_2(t)}}{
           \frac{1}{y_0}+b\int_0^t\frac{1}{k_2+\xu(s)}e^{
             \bigl(b-\frac{\sigma_2^2}{2}  \bigr)s+\sigma_2w_2(s)}ds} .
       \end{equation}
Similarly, we have, for every $t\geq 0$,
\begin{align}
  0&\leq \xdd(t)\leq x(t) \label{eq:comparison-xdd}\text{ a.s.},\\
  0&\leq \ydd(t)\leq y(t) \label{eq:comparison-ydd}\text{ a.s.},
\end{align}
with
\begin{align}
  d\xdd(t)& =\Bigl( \xdd(t)(1-\xdd(t))-a \yu(t) \Bigr)dt
            + \sigma_{1} \xdd(t) dw_{1}(t),
            \quad \xdd(0)=x_0,
            \label{eq:xdd}\\
  d\ydd(t)& =b \ydd(t)\CCO{ 1-\frac{\ydd(t)}{k_2} }dt
            + \sigma_{2} \ydd(t) dw_{2}(t),
            \quad \ydd(0)=y_0.
            \label{eq:ydd}
\end{align}
Note that $\xdd$ is defined with the help of the process $\yu$ defined by
\eqref{eq:yu}.

The following property of stochastic logistic processes will be
useful:
       %%%%%%%%%%%%%%%%%%%%%%%%%%%%%%%%%%%%
       \begin{lemma}\label{lem:limit_logistic}
         (\cite[Theorem 3.2 and Theorem 4.1]{liu-shen2015})
       The process $\xu$ converges a.s.~to $0$ if
	$\sigma_1^2\geq 2$, whereas it converges to a nondegenerate stationary
	distribution if $0<\sigma_1^2< 2$.

        Similarly, $\ydd$ converges a.s.~to $0$ if
	$\sigma_2^2\geq 2b$, whereas it converges to a nondegenerate stationary
	distribution if $0<\sigma_2^2< 2$.
       \end{lemma}

%%%%%%%%%%
\begin{remark}\label{rem:permanence_sandwich}
  The global existence and uniqueness of $(\xu,\yu,\xdd,\ydd)$ can be
  obtained via the same methods as in \Cref{subsec:existence},
  see in particular \Cref{rem:altproof}.
\end{remark}

%%%%%%%%%%%%%%%%%%% PERSISTENCE-EXTINCTION %%%%%%%%%%%%%%%%%%%%%%%%%%%%%%
\subsection{Extinction}
We show that, when the noise is large, the system
\eqref{eq:stochastic} goes almost surely (but in infinite time) to
extinction.

%%%%%%%%%%%%%%% THEO STATIONARY
\begin{theorem}\label{theo:extinction}
Assume that $\sigma_1^2\geq 2$.
Then $\lim_{t\rightarrow\infty}x(t)=0$ a.s.
If moreover
 $\sigma_2^2\geq 2b$, then  $\lim_{t\rightarrow\infty}y(t)=0$ a.s.
\end{theorem}
%%%%%%%%%%%
\begin{proof}
If $\sigma_1^2\geq2$, we deduce from \eqref{eq:xu} and
\Cref{lem:limit_logistic} that $x(t)$ converges to $0$ a.s.

Assume moreover that $\sigma_2^2\geq 2b$.
From \eqref{eq:yu-explicit}, the random variable
$\yu :\,\Omega\rightarrow C(\R_+;\R_+)$ is a function of two independent
random variables, $w_2$ and $\xu$ (the latter is a function of $w_1$).
For a fixed $\xu\in C(\R_+;\R_+)$ such that
$\lim_{t\rightarrow\infty}\xu(t)=0$, we have
\begin{equation}\label{eq:yu-ydd}
  \lim_{t\rightarrow\infty}\bigl(\yu(t)-\ydd(t)\bigr)=0,
\end{equation}
where $\ydd$ is defined by \eqref{eq:xdd}.
Thus, since $\xu(t)$ goes to $0$ a.s., \Cref{eq:yu-ydd} is
satisfied a.s.
Since, by \Cref{lem:limit_logistic},
$\ydd(t)$ converges a.s.~to $0$ if $\sigma_2^2\geq 2b$, we deduce that
$\lim_{t\rightarrow\infty}\yu(t)=0$ a.s.,
and the result follows from \eqref{eq:comparison-yu}.
\end{proof}
%%%%%

%%%%%%%%%%%
\begin{remark}
  Since $\ydd(t)\leq y(t)\leq\yu(t)$, we can deduce also from
  \eqref{eq:yu-ydd} that, if $\sigma_1^2\geq2$ with
$0<\sigma_2^2<2b$, then
$x(t)$ converges a.s.~to $0$ while $y(t)$
converges to a nondegenerate stationary
	distribution.
\end{remark}
%%%%%%%%%%%

%%%%%%%%%%%%%%%%%%%%%%%%%%%%%%%%%%%%%%%%%%%%%%%%%%%%%%%%%%%%%%%%%%%%%%%
\subsection{Existence of a stationary distribution}
In this section, we assume that $m>0$.
The existence of a stationary distribution is proved for a similar
(but different) system without refuge in \cite{fu-jiang-shi-hayat-alsaedi}.

%%%%%%%%%%%%%%% THEO STATIONARY
\begin{theorem}\label{theo:stationary}
Assume that $0<\sigma_1^2<2$ and $0<\sigma_2^2<2b$, with $m>0$.
Then the system \eqref{eq:stochastic}
has a unique stationary distribution $\mu$ on
$]0,+\infty[\times ]0,+\infty[ $.
  Moreover, the system \eqref{eq:stochastic} is ergodic and
  its transition probility $\prob((x,y),t,.)$
  satisfies
\begin{equation*}
  %\label{eq:muconv}
  \prob((x_0,y_0),t,\varphi)\rightarrow\mu(\varphi)
  \text{ when }t\rightarrow\infty
\end{equation*}
for each $(x_0,y_0)\in]0,+\infty[\times ]0,+\infty[$
and each bounded continuous function
$\varphi :\,\allowbreak]0,+\infty[\times ]0,+\infty[\allowbreak\rightarrow\allowbreak\R$.
\end{theorem}
%%%%%%%%%%%
\begin{remark}
  \Cref{theo:stationary} shows that, contrarily to the deterministic case,
  when $\min\{\sigma_1,\sigma_2\}>0$, there is only one equilibrium
  for the system \eqref{eq:stochastic}
  in the open quadrant $]0,+\infty[\times ]0,+\infty[$.

  Note also that, when $\min\{\sigma_1,\sigma_2\}>0$,
  there is no invariant closed subset in the open quadrant
  $]0,+\infty[\times ]0,+\infty[$ for the system
  \eqref{eq:stochastic}. Indeed,
  since the noise in \eqref{eq:stochastic} acts in all directions,
  the viability conditions of \cite{daprato-frankowska}
  are satisfied for no closed convex subset of $]0,+\infty[\times
  ]0,+\infty[$.

In particular, there is no equilibrium point for
\eqref{eq:stochastic}, thus the limit stationary distribution is
nondegenerate.
\end{remark}
%%%%%%%%%%%
\begin{remark}
The ecologically less interesting case
when $(x,y)$ stays in one of the coordinate axes has
similar features, since, by \cite[Theorem 3.2]{liu-shen2015},
the stochastic logistic equation admits a
unique invariant ergodic distribution when the diffusion coefficient is
positive but not too large.
\end{remark}

Our proof of \Cref{theo:stationary}
is based on the following well known result:
\begin{lemma}\label{lem:khasminskii}
  Consider the equation
  \begin{equation}\label{eq:khasminskii}
    dX(t)=f(X(t))\,dt+g(X(t))\,dW(t)
  \end{equation}
  where $f:\,\R^d\rightarrow \R^d$ and $g:\,\R^d\rightarrow
  \R^{m\times d}$ are locally Lipschitz functions with locally
  sublinear growth, and $W$ is a standard Brownian motion on
  $\R^m$. Denote by $A(x)$ the $m\times m$ matrix $g(x)\,g(x)^{T}$.
  Assume that $]0,+\infty[^d$ is invariant by \eqref{eq:khasminskii} and
  that there exists a bounded open subset $U$ of $]0,+\infty[^d $ such that
the  following conditions are satisfied:
\begin{itemize}
\item[\emph{(B.1)}] In
  a neighborhood of $U$, the smallest eigenvalue of $A(x)$ is bounded
  away from $0$,
\item[\emph{(B.2)}] If $x\in\R^d\setminus U$,
  the expectation of the hitting time
  $\tau^U$ at which the solution to  \eqref{eq:khasminskii} starting
  from $x$ reaches the set $U$ is finite, and
  $\sup_{x\in K}\expect^x\tau^U<\infty$ for every compact subset $K$
  of $]0,+\infty[^d$.
\end{itemize}
Then \eqref{eq:khasminskii} has a unique stationary distribution
$\mu$ on $]0,+\infty[^d$.
Moreover, \eqref{eq:khasminskii} is ergodic,
its transition probility $\prob(x,t,.)$
%associated with \eqref{eq:khasminskii}
satisfies
\begin{equation}
  \label{eq:muconv}
  \prob(x,t,\varphi)\rightarrow\mu(f)\text{ when }t\rightarrow\infty
\end{equation}
for each $x\in\R^d$ and each bounded continuous
$\varphi :\,]0,+\infty[^d\rightarrow\R$.
\end{lemma}
%%%%%%%%%
The existence of the stationary distribution
comes from \cite[Theorem 4.1]{khasminski},
its uniqueness from \cite[Corollary 4.4]{khasminski},
the ergodicity from \cite[Theorem 4.2]{khasminski},
and \eqref{eq:muconv}
comes from \cite[Theorem 4.3]{khasminski}.
Section 4.8 of \cite{khasminski} contains remarks that allow the
restriction to an invariant domain such as $]0,+\infty[^d $.

To prove Condition (B.2), we establish
some preliminary results using the systems
\eqref{eq:xu}-\eqref{eq:yu} and \eqref{eq:xdd}-\eqref{eq:ydd} of
\Cref{subsec:comparison}.
Let us first set some notations:
%%%%%%%%%%%%%%%%%%%%%%%%%%%%%%%%%%%%%
For $r,R,x_0,y_0>0$, we denote
\begin{align*}
  \txu{R}(x_0)=&\inf\{t\geq 0\tq \xu(t)< R\},\\
  \tyu{R}(x_0,y_0)=&\inf\{t\geq 0\tq \yu(t)< R\},\\
  \txdd{r}(x_0)=&\inf\{t\geq 0\tq x(t)> r\},\\
  \tydd{r}(y_0)=&\inf\{t\geq 0\tq \ydd(t)> r\},
\end{align*}
where $\inf\emptyset=+\infty$,
$\xu$, $\yu$, and $\ydd$ are the solutions to
\eqref{eq:xu},
\eqref{eq:yu}, and
\eqref{eq:ydd} respectively, and $x$ is the first component of the solution to
\eqref{eq:stochastic} starting from $(x_0,y_0)$.
Note that, since $\yu$ depends on $\xu$, the hitting time $\tyu{R}$
depends on $(x_0,y_0)$.

%%%%%%%%%%%%%%%%%%%%%%%%%%%%%%%%%%%%%
Since \eqref{eq:xu} and \eqref{eq:ydd} are stochastic logistic equations,
the proof of \cite[Theorem 3.2]{liu-shen2015} shows the following:
%%%%%%%%%%%
\begin{lemma}\label{lem:hitting-xu-ydd}
Assume that $0<\sigma_1^2<2$. There exists $R_1>0$ sufficiently large such that
$\expect\CCO{\txu{R_1}(x_0)}$ is finite and uniformly bounded on compact subsets
of $[R_1,+\infty[$.

Assume that $0<\sigma_2^2<2b$.
There exists $r_2>0$ sufficiently small such that
$\expect\CCO{\tydd{r_2}(y_0)}$ is finite and uniformly bounded on
compact subsets of $]0,r_2]$.
\end{lemma}
Note that the proof of \cite[Theorem 3.2]{liu-shen2015} provides a
two-sided version of \Cref{lem:hitting-xu-ydd} (that is, each of
the processes $\xu$ and $\ydd$ hits an interval of the form $]r,R[$
in finite time), but we only need the one-sided version stated here.

%%%%%%%%%
\begin{lemma}\label{lem:hitting-xdd}
  Assume that $0<\sigma_1^2<2$.
There exists $r_1$ sufficiently small such that
$\expect(\txdd{r_1}(y_0))$ is finite and uniformly bounded on compact subsets
of $]0,r_1]$.
\end{lemma}
%%%%%%%%%
\begin{proof}
We use the fact that, when $x<m$, $x$ coincides with a process $z$
solution to the stochastic logistic equation
\begin{equation*}
  dz(t)=z(1-z)dt+\sigma_1zdw_1(t).
\end{equation*}
The proof of \cite[Theorem 3.2]{liu-shen2015} provides a number $r>0$
such that the expectation of the hitting time of $]r,+\infty[$ by $z$
is finite and uniformly bounded on each compact subset of $]0,r]$.
Then, we only need to take $r_1=\min\{r,m\}$.
\end{proof}

%%%%%%%%%%%%%%%%%%%%%%%%%%%%%%%%%%%%%
\begin{lemma}\label{lem:hitting-yu}
There exists $R_2$ sufficiently large such that
$\expect(\tyu{(R_2)}(x_0,y_0))$ is finite and uniformly bounded on
compact subsets
of $]0,+\infty[\times [R_2,+\infty[$.
\end{lemma}
%%%%%%%%%
\begin{proof}
Let us set, for $u,v>0$,
\begin{equation*}
  V(u,v)=\frac{1}{u}+u+\frac{1}{v}+\log(v).
\end{equation*}
We have $V(u,v)\geq V(1,1)>0$.
Let $L$ be the infinitesimal operator (or Dynkin operator) of the
system \eqref{eq:xu}-\eqref{eq:yu}. We have
\begin{align*}
  LV(\xu,\yu)=
                  & \xu(1-\xu)\CCO{1-\frac{1}{\xu^2}}
                        +\frac{\sigma_1^2}{\xu}\\
                  & +b\yu \CCO{ 1-\frac{\yu}{k_2+\xu} }
                    \CCO{-\frac{1}{\yu^2}+\frac{1}{\yu} }
                    +\frac{\sigma_2^2}{2}\CCO{\frac{2}{\yu}-1 }\\
  =&-\frac{1+\xu}{\xu}\bigl((\xu-1)^2-\sigma_1^2 \bigr)-\sigma_1^2\\
   & +b\frac{\yu-1}{\yu}\,\frac{k_2+\xu-\yu}{k_2+\xu}
     +\frac{\sigma_2^2}{2}\frac{2-\yu}{\yu}.
\end{align*}
Let $\rho\geq 1$ such that
\begin{equation*}
  \xu>\rho\Rightarrow (\xu-1)^2-\sigma_1^2 >b\xu.
\end{equation*}
For $\xu>\rho$ and $\yu>\max\{4,1/(b+\sigma_1^2)\}$, we get
$(2-v)/v\leq -1/2$ and
\begin{equation*}
  LV(\xu,\yu)
  \leq -\frac{1+\xu}{\xu}\,b\xu-\sigma_1^2
    +b-\frac{\sigma_2^2}{4}%\\
  =-b\xu-\sigma_1^2-\frac{\sigma_2^2}{4}
  \leq-\frac{b+\sigma_1^2}{b+\sigma_1^2}-\frac{\sigma_2^2}{4}
        <-1.
\end{equation*}
On the other hand, there exists a number $K\geq 0$ such that
\begin{equation*}
  \xu\leq \rho\Rightarrow (\xu-1)^2-\sigma_1^2 \leq K.
\end{equation*}
For $\xu\leq\rho$ and $\yu\geq \max\{4,(1+2/b)(k_2+\rho)\}$, we have
$(v-1)/v\geq 3/4$ and
$(k-2+\rho-v)/(k_2+\rho)\leq -2/b$, thus
\begin{align*}
 LV(\xu,\yu)
 \leq &-\frac{1+\rho}{\rho}K -\sigma_1^2
 +b\frac{\yu-1}{\yu}\,\frac{k_2+\rho-\yu}{k_2+\rho}
        -\frac{\sigma_2^2}{4}\\
  \leq & -\frac{1+\rho}{\rho}K -\sigma_1^2
  +b\times\frac{3}{4}\times\frac{-2}{b}-\frac{\sigma_2^2}{4}
 \\
 < &-1.
\end{align*}
Let $R_2=\max\{4,1/(b+\sigma_1^2),(1+2/b)(k_2+\rho)\}$.
For every $y_0>R_2$ and every $x_0>0$, we have $LV(\xu,\yu)<-1$.
Denote for simplicity $\tau=\tyu{R_2}(x_0,y_0)$.
We have
\begin{multline*}
  0\leq \expect^{(x_0,y_0)}V(\xu(\tau),\yu(\tau))\\
  =V(x_0,y_0)+\expect^{(x_0,y_0)}\int_0^{\tau}LV(\xu(s),\yu(s))ds
  \leq V(x_0,y_0)-\expect(\tau),
\end{multline*}
which proves that $\expect(\tau)\leq V(x_0,y_0)<\infty$.
\end{proof}

%%%%%%%%%%%%
%\preuvof{Theorem \cref{theo:stationary}}
\begin{proof}[Proof of \Cref{theo:stationary}]
Condition (B.1) of \Cref{lem:khasminskii} is trivially
statisfied.

To prove Condition (B.2), with the notations of
\Cref{lem:hitting-xu-ydd,lem:hitting-xdd,lem:hitting-yu},
taking into account
the inequalities \eqref{eq:comparison-xu},
\eqref{eq:comparison-yu}, and
\eqref{eq:comparison-ydd},
we only need to take $r$ and $R$ such that $0<r<R$, $r\leq
\min\{r_1,r_2\}$,
$R\geq\max\{R_1,R_2\}$, and
$U=]r,R[\times]r,R[$.
\end{proof}
%%%%

%%%%%%%%%%%%%%%%%%%%%%%%%%%%%%%%%%%%%%%%%%%%%%%%%%%%%%%%%%%%%%%%%%%%%%%%%%%%%%%
\section{Numerical simulations and figures}\label{sec:simul}
All simulations and pictures of this section are obtained using Scilab.
	\subsection{Deterministic system}
	We numerically simulate solutions to System
        \eqref{eq:simple}. Using the Euler scheme, we consider the
        following discretized system:
	\begin{equation}
		\begin{aligned}
         x_{k+1}=&\ x_{k}+\croche{x_{k}(1-x_{k})-\frac{a y_{k}(x_{k}-m)}{k_{1}+x_{k}-m}}h ,\\
		 y_{k+1}=&\,y_{k}+b y_k \croche{1-\frac{y_k}{k_{2}+x_{k}-m} }h.		
			\end{aligned}
                      \end{equation}
         Simulations are shown in \Cref{fig:3points-cycle,Fig1cycle1pt,hopf}.
	%%%%%%%%%%%%%%%%%%%%% 1. FIGURES 3 POINTS %%%%%%%%%%%%
	\begin{figure}[!h]%[tbhp]
          \centering
	%	\begin{multicols}{2}
		\includegraphics[width=1\textwidth]{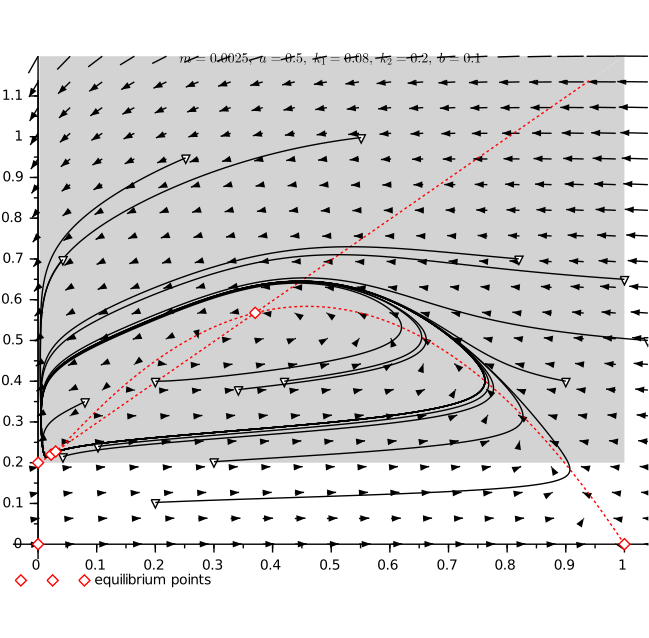}
			\caption{A phase portrait of \eqref{eq:simple}
                          with three equilibrium points and a cycle
			in the interior of $\invrg$.
                      The dashed lines are isoclines
			$y=\frac{x(1-x)(k_1+x-m)}{a(x-m)}$
 			and $y=k_2+x-m$.
                The grey region is the invariant attracting domain $\invrg$.\\
                      $m=0.0025$, $a=0.5$, $k_1=0.08$, $k_2=0.2$, $b=0.1$.}
\label{fig:3points-cycle}%% label apres caption sinon pas pris en compte
						\end{figure}

%%%%%%%%%%%%%%%%%%%%% 2. CYCLE AVEC PT INSTABLE %%%%%%%%%%%%%%%%%%%%%%%%%%%%%%%%%%
	\begin{figure}[!h]%[tbhp]%[ht!]%[tbhp]
	      \includegraphics[width=1\textwidth]{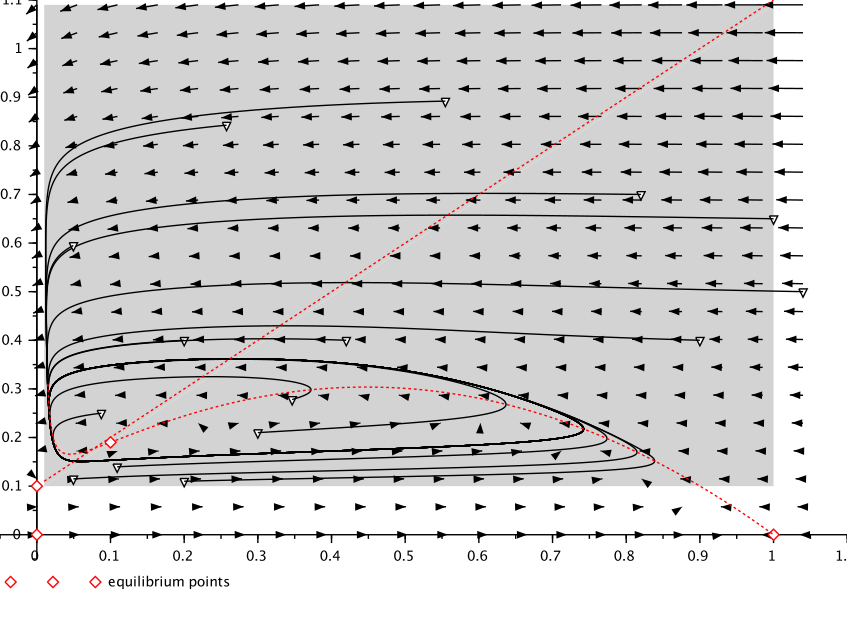}
		\caption{A phase portrait of \eqref{eq:simple} with an
                  unstable equilibrium and a stable limit cycle.\\
                $m=0.01$, $a=1$, $k_1=0.1$, $k_2=0.1$, $b=0.05$.}
		\label{Fig1cycle1pt}
	\end{figure}
%%%%%%%%%%%%  3. HOPF %%%%%%%%%%%%%%%%%%%%%%%%%%%%%%%%%%%%%%%%%%%%%%
%  \begin{center}
	\begin{figure}[!h]%[tbhp]%[ht!]
\centering
\begin{subfigure}{1\linewidth}%\label{fig:a}
    \includegraphics[width=1\textwidth]{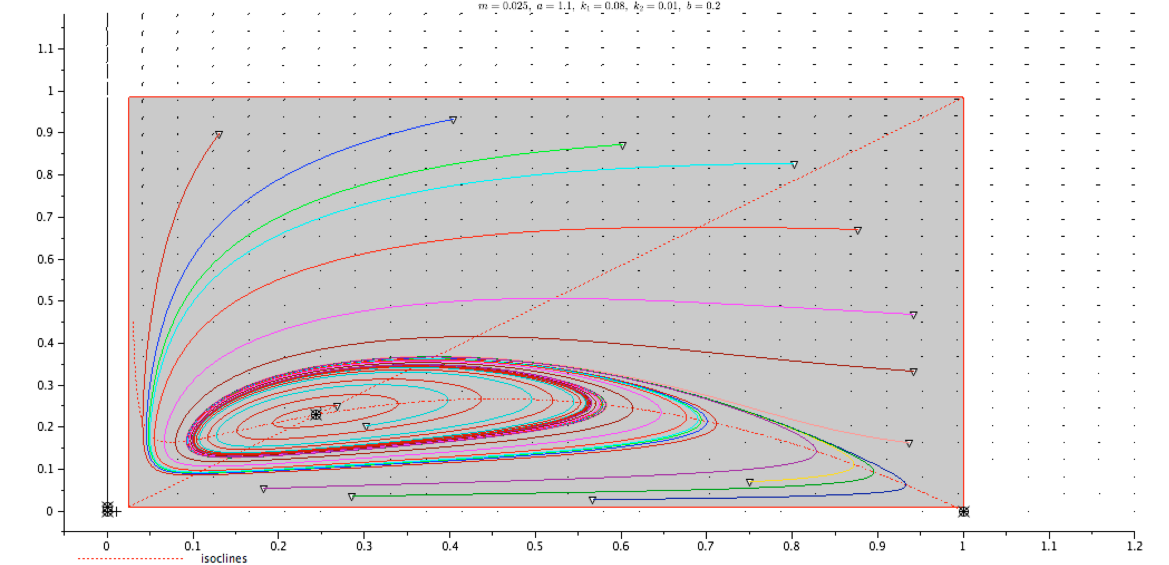}
    \caption{$\lambda<0$ (semi hyperbolic case):
    $m=0.0025$, $a=1.1$, $k_1=0.08$, $k_2=0,01$, $b=0.2$.}
\end{subfigure}
\begin{subfigure}{1\linewidth}%\label{fig:b}
    \includegraphics[width=1\textwidth]{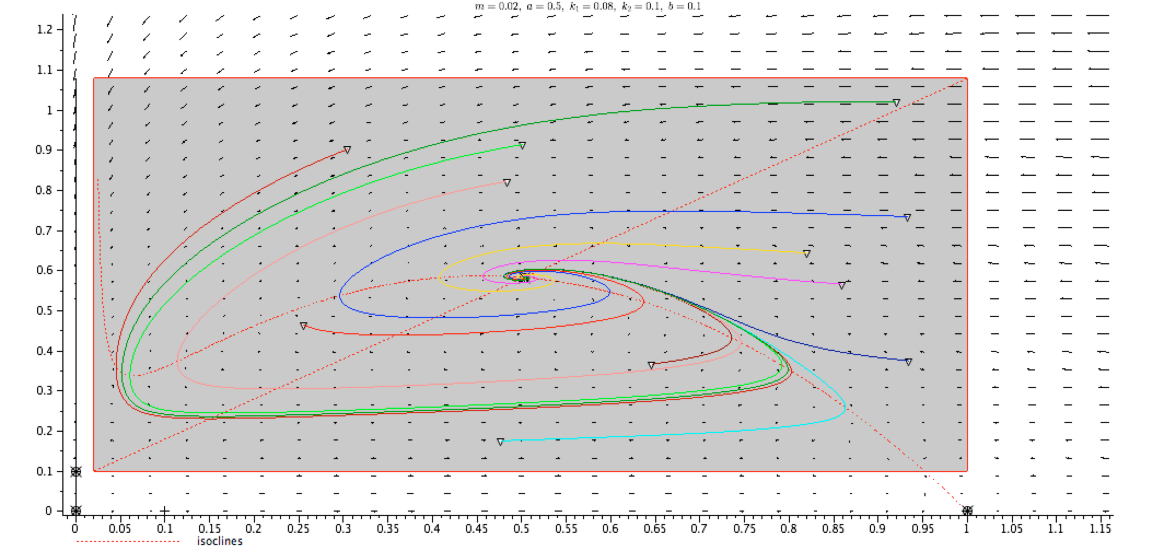}
    \caption{$\lambda>0$ (semi hyperbolic case):
    $m=0.002$, $a=0.5$, $k_1=0.08$, $k_2=0.1$, $b=0.1$.}
  \end{subfigure}
  \caption{Hopf bifurcation of the system \eqref{eq:simple}.
          }
  \label{hopf}
  %% label apres caption sinon pas pris en compte
	\end{figure}

%%%%%%%%%%%%%%%%%%%%%%%%%%%%%%%%%%%%%%%%%%%%%%%%%%%
	\subsection{Stochastically perturbated system}
	We numerically simulate the solution to System
        \eqref{eq:stochastic}. Using the Milstein scheme (see
        \cite{kloeden-platen}), we consider the discretized
        system
	\begin{equation}
		\begin{aligned}
         x_{k+1}=&\ x_{k}+\croche{x_{k}(1-x_{k}-\frac{a y_{k}}{k_{1}+x_{k}-m})}h
         + \sigma_{1}\, x_{k} \sqrt{h}\, \xi^2_{k}+\frac{1}{2}\sigma_{1}^2 x_{k}(h \,\xi^2_{k}-h ) ,\\
		 y_{k+1}=&\,y_{k}+b y_k \croche{1-\frac{y}{k_{2}+x} }h+\sigma_{2} y_{k} \sqrt{h}\, \xi^2_{k}+\frac{1}{2}\sigma_{2}^2 y_{k}(h \,\xi^2_{k}-h ),		
			\end{aligned}
                      \end{equation}
       where $(\xi_k)$ is an i.i.d.~sequence of normalized centered
       Gaussian variables.
       
       Simulations of the stochastically perturbated case are shown in
       \Cref{fig:stoch}.
	% In \Cref{fig:stoch},
        % we choose $a=0.4$, $k_{1}=0.08$, $k_{2}=0.2$, $b=0.1$,
        % $m=0.0025$, the initial value
        % $(x(0), y(0))=(0.55,0.6),$ and the time step $h=0.01.$ The
        % deterministic model has a globally stable equilibrium point
        % $(x^*,y^*)=(0.55,0.75)$.
These simulations  show the permanence
        of the system \eqref{eq:stochastic}.
%%%%%%%%%%%%%%%%%%%%%%%%%%%%%%%%%%%%%%%%%%%%%%%%%%%%%%%%%%%%%%%%%%%%%%%%%%%%
        \begin{figure}[!h]%[tbhp]%[ht!]%[H]
          \centering
  \begin{subfigure}{.8\linewidth}%\label{fig:a}
    \includegraphics[width=.8\textwidth]{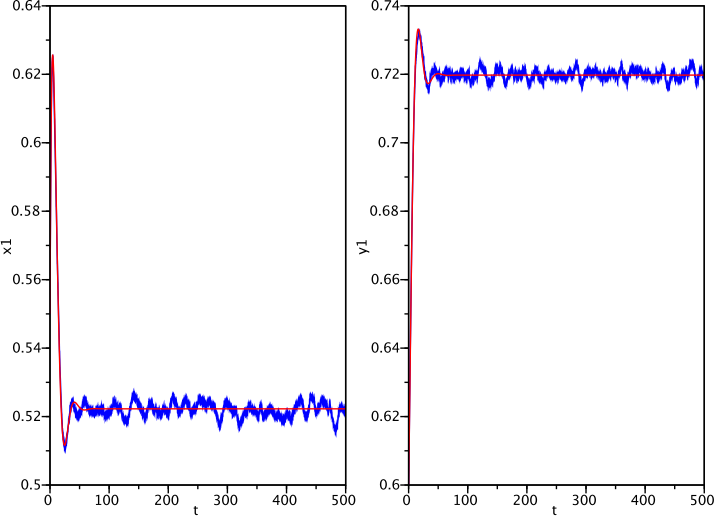}%\\
\caption{$\sigma_{1}=0.01, \sigma_{2}=0.01$}
\end{subfigure}
%%%%%%%%%%%%%%%%%
\begin{subfigure}{.8\linewidth}%\label{fig:b}
  \includegraphics[width=.8\textwidth]{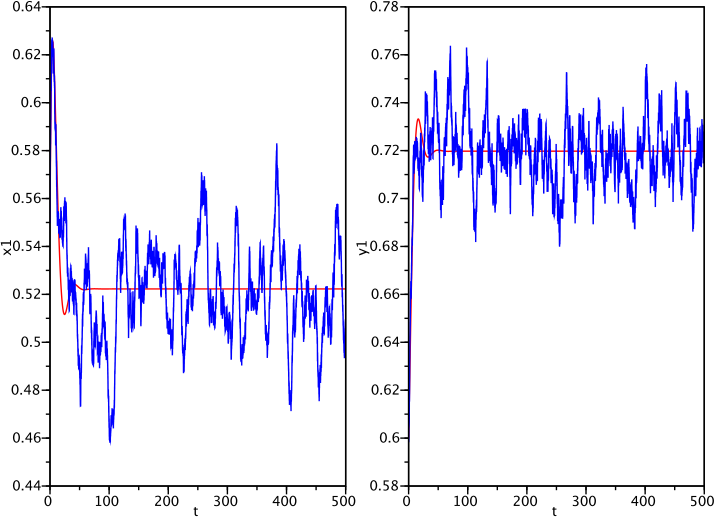}
\caption{$\sigma_{1}=0.3$, $\sigma_{2}=0.2$}
\end{subfigure}

			\caption{Solutions to the stochastic system
          \eqref{eq:stochastic} and the corresponding deterministic
          system, represented respectively by the blue line and the
          red line. \\
$a=0.4$, $k_{1}=0.08$, $k_{2}=0.2$, $b=0.1$,
        $m=0.0025$, the initial value
        $(x(0), y(0))=(0.55,0.6),$ and the time step $h=0.01.$ The
        deterministic model has a globally stable equilibrium point
        $(x^*,y^*)=(0.55,0.75)$.}
        \label{fig:stoch}
 \end{figure}		
%%%%%%%%%%%%%%%%%%%%%%%%%%%%%%%%%%%%%%%%%%%%%%%%%%%%%%%%%%%%%%%%%%%%%%%%%%%	
 \FloatBarrier %from package placeins
 \section*{Acknowledgments} We thank an anonymous referee for his
 useful comments.
	%%%%%%%%%%%%%%%%%%%%%%%%%%%%%%%%%%%%%%%%%%%%%%%%
	%\bibliographystyle{AIMS}\bibliography{Holling2}
	%%%%%%%%%%%%%%%%%%%%%%%%%%%%%%%%%%%%%%%%%%%%%%%%
        %%\printbibliography
\def\cprime{$'$}
\providecommand{\href}[2]{#2}
\providecommand{\arxiv}[1]{\href{http://arxiv.org/abs/#1}{arXiv:#1}}
\providecommand{\url}[1]{\texttt{#1}}
\providecommand{\urlprefix}{URL }

\textit{email address}: {slimani\_safia@yahoo.fr}\\
\textit{email address}: {prf@univ-rouen.fr}\\
\textit{email address}: {islam.boussaada@l2s.centralesupelec.fr }
 % \email{slimani\_safia@yahoo.fr}% ORCID 0000-0002-3884-5821
 % \email{prf@univ-rouen.fr}% ORCID 0000-0001-5527-9393
 % \email{islam.boussaada@l2s.centralesupelec.fr }
	 
\end{document}